\documentclass[a4paper,11pt]{article}
  \usepackage[T1]{fontenc}
 \usepackage[normalem]{ulem}
 \usepackage{verbatim}
 \usepackage{graphicx}
 \usepackage{amsmath}
 \usepackage{amssymb}
\usepackage{geometry}
\usepackage{amsthm}
\usepackage[utf8]{inputenc}
\usepackage{esint}

\usepackage{amsmath,amssymb,amsfonts,amsthm}
\usepackage{hyperref}
 \usepackage[boxed]{algorithm2e}
\usepackage{latexsym}
\usepackage{pdflscape}
\usepackage{examplep}
\usepackage{longtable}
\usepackage{graphicx}
 \usepackage{psfrag}
 \usepackage{epsfig}
\usepackage{booktabs}
\usepackage{memhfixc}
\usepackage{color}
\usepackage{tcolorbox}
\usepackage{enumitem}
\usepackage{dsfont}
\usepackage{hyperref}
\usepackage{epstopdf}
\hypersetup{colorlinks=true, linkcolor =red, citecolor = blue}

\usepackage{subfig}

\newtheorem{theorem}{Theorem} 

\newtheorem{prop}{Proposition}
\newtheorem{lem}{Lemma}
\newtheorem{remark}{Remark}

\DeclareMathOperator{\dive}{div}

\def\N{\mathbb{N}}
\def\Z{\mathbb{Z}}
\def\cJ{\mathcal{J}}

\def\cE{\mathcal{E}}

\def\cC{\mathcal{C}}
\def\cD{\mathcal{D}}

\def\dps{\displaystyle}

\def\R{\mathbb{R}}
\def\1{\mathds{1}}

\title{Cross-diffusion systems with non-zero flux and moving boundary conditions.}
\date{}

\author{
  Athmane Bakhta \thanks{Universit\'e Paris-Est, CERMICS(ENPC)}\\
  \texttt{athmane.bakhta@cermics.enpc.fr} 
\and
  Virginie Ehrlacher  \thanks{Universit\'e Paris-Est, CERMICS (ENPC) \& INRIA (Matherials team-project)}\\
  \texttt{virginie.ehrlacher@enpc.fr}
  \and
 }

\begin{document}
\maketitle

\begin{abstract}
We propose and analyze a one-dimensional multi-species cross-diffusion system with non-zero-flux boundary conditions on a moving domain, motivated by the modeling of a Physical Vapor Deposition process. Using the boundedness 
by entropy method introduced and developped in~\cite{Burger,Jungel}, we prove the existence of a global weak solution to the obtained system. In addition, existence of a solution to an  {optimization} problem defined on the fluxes is established under the assumption that the solution to the considered cross-diffusion system is unique. Lastly, 
we prove that in the case when the imposed external fluxes are constant and positive and the entropy density is defined as a classical logarithmic entropy, the concentrations of the different species converge 
in the long-time limit to constant profiles at a rate inversely proportional to time. These theoretical results are illustrated by numerical tests.   
\end{abstract}

%\tableofcontents

\section*{Introduction}
\label{sec:intro}

The aim of this work is to propose and analyze a mathematical model for the description of a Physical Vapor Deposition 
(PVD) process {, the different steps of which are described in details for instance in}~\cite{PVD}. Such a technique is used in several contexts, for instance for the fabrication 
of thin film crystalline solar cells. The procedure works as follows: a wafer is introduced in a hot chamber 
where several chemical elements are injected under a gaseous form. As the latter deposit on the substrate, 
an heterogeneous solid layer grows upon it. Two main phenomena have to be taken into account: the first is naturally the 
evolution of the surface of the film; the second is the diffusion of the various species in the bulk, due to the high temperature conditions. Experimentalists 
are interested in controlling the external gas fluxes that are injected into the chamber, so that, at the end of the process, the spatial distributions of the concentrations of the diverse components
inside the new layer are as close as possible to target profiles.

\medskip

In this article, a one-dimensional model which takes into account these two factors is studied. We see this work as a preliminary step before tackling more challenging models 
in higher dimensions, including surfacic diffusion effects for instance. This will be the object of future work. Our main motivation for the study of such a model 
concerns the optimization of the external fluxes injected in the chamber during a PVD process.

\medskip

More precisely, let us assume that at a time $t\geq 0$, the solid layer is composed of $n+1$ different 
chemical species and occupies a domain of the form 
$(0, e(t)) \subset \R_+$, where $e(t)>0$ denotes the thickness of the film. The evolution of $e(t)$ is determined by the fluxes of atoms that are absorbed at the 
surface of the layer. At time $t>0$ and point $x\in (0,e(t))$, the local  {volumic fractions} of the different species 
are denoted respectively by $u_0(t,x), \cdots, u_n(t,x)$.  {Let us point out that if the molar volume of the solid is uniform in the thin film layer and constant during all the process, then $u_i(t,x)$ is also equal (up to a constant multiplicative constant) to the 
local concentration of the $i^{th}$ species at time $t>0$ and point $0\leq x\leq e(t)$. }Up to some renormalization condition, it is natural to expect 
that these functions are non-negative and satisfy a volumic constraint which reads as follows: 
\begin{equation}\label{eq:volumic}
\forall 0\leq i \leq n, \quad u_i(t,x)\geq 0 \mbox{ and } \sum_{i=0}^n u_i(t,x) = 1.
\end{equation}
 {Because of the constraint~(\ref{eq:volumic}), it holds that $u_0(t,x) = 1 - \sum_{i=1}^n u_i(t,x)$ for all $t>0$ and $x\in (0, e(t))$. Thus, the knowledge 
of the $n$ functions $u_1,\cdots,u_n$ is enough to determine the dynamics of the whole system. Replacing $u_0$ by $1-\sum_{i=1}^n u_i$, and} denoting by $u$ the vector-valued function $(u_1, \cdots, u_n)$, the evolution of the concentrations inside 
the bulk of the solid layer is modeled through a system of cross-diffusion equations
of the form
\begin{equation}\label{eq:cross-diffusion}
\partial_t u - \partial_x\left( A(u) \partial_x u \right) = 0, \quad \mbox{for }t> 0, \; x\in (0, e(t)),
\end{equation}
with approriate boundary and initial conditions, where $A: [0,1]^n \to \R^{n\times n}$ is a matrix-valued 
function encoding the cross-diffusion properties of the different species. 

\medskip

Such systems have received much attention from the mathematical 
community in the case when no-flux boundary conditions are imposed on a fixed domain~\cite{Ladyzenskaia, Amann, LeNguyen, Griepentrog}. 
Then, in arbitrary dimension $d\in\N^*$, the system reads 
$$
\partial_t u - \mbox{\rm div}_x\left( A(u) \nabla_x u \right) = 0, \quad \quad \mbox{for }t> 0, \; x\in \Omega,
$$
for some fixed bounded regular domain $\Omega \subset \R^d$ and boundary conditions 
$$
\left(A(u) \nabla_x u\right) \cdot \textbf{n} = 0 \mbox{ on } \partial \Omega \mbox{ and } t\geq 0, 
$$
where $\textbf{n}$ denotes the outward normal unit vector to $\Omega$. 

Such systems appear naturally in the study of population's dynamics in biology, and in chemistry, for the study of the evolution of 
chemical species concentrations in a given environment~\cite{Painter1, Painter2}. The analysis of these systems is a challenging task from a mathematical point of 
view~\cite{Lepoutre, Alikakos, Kufner, Redlinger, ChenJungel1,ChenJungel2, Difrancesco}. 
Indeed, the obtained system of parabolic partial differential equations 
may be degenerate and the diffusion matrix $A$ is in general not symmetric and/or not positive definite. Besides, in general, no maximum principle can be proved 
for such systems. Nice counterexamples are given in~\cite{Stara}: there exist H\"older continuous solutions to certain cross-diffusion systems which are not bounded, and there exist bounded 
weak solutions which develop singularities in finite time. 

It appears that some of these cross-diffusion systems have a formal gradient flow structure. Recently, an elegant idea, which consists in introducing 
an entropy density that appears to be a Lyapunov functional for these systems, has been introduced by Burger et al. in~\cite{Burger}. This analysis strategy, which was later
extended by J\"ungel in~\cite{Jungel} and named \itshape boundedness by entropy \normalfont technique, 
enables to obtain the existence of global in time weak solutions satisfying (\ref{eq:volumic}) under 
suitable assumptions on the diffusion matrix $A$. It was successfully applied in several contexts 
(see for instance~\cite{JungelStelzer1, JungelStelzer2, JungelZamponi1, JungelZamponi2}).

\medskip

However, there are very few works which focus on the analysis of such cross-diffusion systems with 
non zero-flux boundary conditions and moving domains. 
To our knowledge, only systems containing at most two different species have been studied,
so that $n=1$ and the evolution of the concentrations inside the domain are decoupled and 
follow independent linear heat equations~\cite{moving}.

\medskip

The one-dimensional model~(\ref{eq:cross-diffusion}) we propose and analyze in this paper describes
the evolution of the concentration of $n+1$ different atomic species, with external flux boundary conditions, in the case when the 
diffusion matrix $A$ satisfies similar assumptions to those needed in the no-flux boundary conditions case studied in~\cite{Jungel}. 

The article is organized as follows: the results of~\cite{Jungel} 
in the case of no-flux boundary conditions in arbitrary dimension are recalled in Section~\ref{sec:noflux}. 
We illustrate them on a prototypical example of diffusion matrix $A$, which is introduced in Section~\ref{sec:example}. 

Our results in the case of a one-dimensional moving domain with non-zero flux boundary conditions are gathered in 
Section~\ref{sec:flux}. We prove the existence of a global in time weak solution to~(\ref{eq:cross-diffusion}) 
with appropriate boundary conditions and evolution law for $e(t)$ in Section~\ref{sec:existence}. The long time 
behaviour of a solution is analyzed in the case of constant external absorbed fluxes in~\ref{sec:longtime} 
and an optimization problem is studied in~\ref{sec:control}. The proofs of these results 
are gathered in Section~\ref{sec:proofs}.

A numerical scheme used to approximate the solution of such systems is described in Section~\ref{sec:numerics} 
and our theoretical results will be illustrated by several numerical tests. We refer the reader to~\cite{theseAthmane} 
for comparisons between our proposed model and experimental results obtained in the context of thin film solar cells fabrication.

%%%%%%%%%%%%%%%%%%%%%%%%%%%%%%%%%%%%%%%%%%%%%%%%%
%%%%%%%%%%%%%%%%%%%%%%%%%%%%%%%%%%%%%%%%%%%%%%%%%
\section{Case of no-flux boundary conditions in arbitrary dimension} \label{sec:noflux}

In Section~\ref{sec:example}, a particular cross-diffusion model on a fixed domain with no-flux boundary conditions is presented. The latter is a prototyical example of the 
systems of equations considered in this paper. 
Its formal gradient flow structure is highlighted in Section~\ref{sec:gradientflow}. 
Using slight extensions of results of~\cite{JungelZamponi1, JungelZamponi2}, 
it can be seen that this system can be analyzed using the theoretical framework developped 
in~\cite{Jungel,Burger}, which is recalled in Section~\ref{sec:Jungel}.

\medskip

Throughout this section, let us denote by $d \in \N^*$ the space dimension, 
$\Omega \subset \R^d$ the regular bounded domain occupied by the solid.
The local concentrations at time $t>0$ and position $x\in \Omega$ of the $n+1$ different 
atomic species entering in the composition of the material are 
respectively denoted by $u_0(t,x),\cdots,u_n(t,x)$. We also denote by $\textbf{n}$ the normal unit vector 
pointing outwards the domain $\Omega$. 

\subsection{Example of cross-diffusion system}\label{sec:example}

\subsubsection{Presentation of the model}\label{sec:model}

As mentioned above, we have one particular example of system of cross-diffusion equations in mind, which is used to illustrate more general 
theoretical results. This system, with no-flux boundary conditions, reads as follows : for any $ 0 \leq i \leq n$, 
\begin{equation}\label{eq:system_all}
\left\{
\begin{array}{ll}
\partial_t u_i - \dive_x \left(  \sum \limits_{0 \leq j \neq i \leq n} K_{ij} (u_j \nabla_x u_i - u_i \nabla_x u_j)\right) = 0,  & \mbox{for }(t,x) \in \R_+^* \times \Omega, \\
 \left(  \sum \limits_{0 \leq j \neq i \leq n} K_{ij} (u_j \nabla_x u_i - u_i \nabla_x u_j)\right)\cdot \textbf{n} = 0,  & \mbox{for }(t,x) \in \R_+^* \times \partial\Omega, \\
\end{array}
\right .
\end{equation}
where for all $ 0 \leq i \neq j \leq n$, the positive real numbers $K_{ij}$ satisfy $K_{ij} = K_{ji} > 0$. They represent the cross-diffusion coefficients of atoms of type $i$ with atoms
of type $j$. This set of equations can be formally derived from a discrete stochastic lattice hopping model, which is detailed in the Appendix. 
\medskip

The initial condition $(u^0_0,\cdots, u_n^0)\in L^1(\Omega ; \R^{n+1})$ of this system is assumed to satisfy:
\begin{equation}\label{eq:initconstraint}
\forall 0\leq i \leq n, \quad u^0_i(x)\geq 0, \quad \sum_{i=0}^n u^0_i(x) = 1 \mbox{ and } u_i(0,x) = u_i^0(x) \quad \mbox{ a.e. in } \Omega . 
\end{equation}
The relationship $\sum_{i=0}^n u^0_i(x) = 1$ is a natural volumic constraint which encodes the fact that each site of the crystalline lattice of the solid has to be occupied 
(vacancies being treated as a particular type of atomic species). 
 
\medskip 

Summing up the $n+1$ equations of~(\ref{eq:system_all}), we observe that a solution $(u_0, \cdots, u_n)$ must necessarily satisfy 
$\partial_t \left( \sum_{i=0}^n u_i \right) = 0$. It is thus expected that the following relationship should hold: 
\begin{equation}\label{eq:constraint}
\forall 0\leq i \leq n, \quad u_i(t,x)\geq 0, \quad \sum_{i=0}^n u_i(t,x) = 1, \quad \mbox{a.e. in } \R_+^* \times \Omega.
\end{equation}

\medskip

Plugging the expression $u_0(t,x) = 1 - \sum_{i=1}^n u_i(t,x)$ in~\eqref{eq:system_all},  { it holds that for all $1\leq i \leq n$,}
\begin{align*}
  {0} &  { = \partial_t u_i - \mbox{\rm div}_x\left[ \sum_{1\leq j \neq i \leq n} K_{ij}\left(u_j \nabla_x u_i - u_i \nabla_x u_j \right)\right]}\\
 &  { - \mbox{\rm div}_x\left[ K_{i0}\left( \left( 1 - \sum_{1\leq j\neq i \leq n} u_j -u_i\right)\nabla_x u_i -u_i \nabla_x \left(1 - \sum_{1\leq j\neq i \leq n} u_j -u_i\right)\right)\right]}\\ 
 &  { = \partial_t u_i  - \mbox{\rm div}_x\left[ \sum_{1\leq j \neq i \leq n} (K_{ij} - K_{i0}) \left(u_j \nabla_x u_i - u_i \nabla_x u_j \right) + K_{i0} \nabla_x u_i \right].}\\ 
 \end{align*}
 {Thus, }the system can be rewritten as a function of $u:= (u_1, \cdots, u_n)^T$ as follows 
\begin{equation}\label{eq:system_part}
\left\{
\begin{array}{ll}
\partial_t u - \dive_x \left( A(u) \nabla_x u \right) = 0, &  \quad \mbox{ for } (t,x) \in \R_+^* \times \Omega,\\
(A(u) \nabla_x u)\cdot \textbf{n} = 0, & \quad \mbox{ for }(t,x)\in \R_+^* \times \partial \Omega,\\
u(0,x) = u^0(x), &\quad \mbox{ for } x \in \Omega,\\
\end{array}
\right .
\end{equation}
where $u^0:=(u_1^0, \cdots, u_n^0)^T$ and the matrix-valued application 
$$ 
A:\left\{ 
					      \begin{array}{ccc}
					      [0,1]^n & \to & \R^{n\times n}\\
					      u:=(u_i)_{1\leq i \leq n} & \mapsto & \left( A_{ij}(u)\right)_{1\leq i,j \leq n}
					      \end{array}
					      \right .
					      $$ 
is defined by
\begin{equation}\label{eq:defA}
\left\{
 \begin{array}{l}
  \forall 1\leq i \leq n, \quad A_{ii}(u) =  \sum \limits_{ 1\leq j \neq i \leq n} (K_{ij}- K_{i0}) u_j + K_{i0}, \\
  \forall 1\leq i\neq j \leq n, \quad A_{ij}(u) =  -(K_{ij}- K_{i0}) u_i.\\
 \end{array}
\right .
\end{equation}

Despite their importance in chemistry or biology, it appears that the mathematical analysis of systems of the form~\eqref{eq:system_part}, 
taking into account constraints~\eqref{eq:constraint}, is quite 
recent~\cite{Burger,Griepentrog,Jungel,Mielke}.  {Let us point out here that the non-negativity of the solutions to~(\ref{eq:system_part}) through time is a mathematical issue, linked to the absence of a maximum principle for such systems.}

At least up to our knowledge, the first proof of existence of global weak solutions of~\eqref{eq:system_part} satisfying constraints \eqref{eq:constraint}
with non-identical cross-diffusion coefficients is given in~\cite{Burger} for $n=2$ with coefficients $K_{ij}$ such that $K_{i0}>0$ for $i=1,2$ and $K_{12} = K_{21} = 0$. 
These results were later extended in~\cite{JungelZamponi2} to a general number of species 
$n\in \N
^*$ with cross-diffusion coefficients satisfying $K_{i0} >0$ and $K_{ij} = 0$ for all $1\leq i \neq j \leq n$; the authors of the latter article proved in addition the uniqueness 
of such weak solutions. In~\cite{JungelZamponi1}, the case $n=2$ with arbitrary positive coefficients $K_{ij}>0$ is covered, though no uniqueness result is provided. 
The main difficulty of the mathematical analysis of such equations relies in the bounds~\eqref{eq:constraint}, which are not obvious since no maximum principle 
can be proved for these systems in general. In all the articles mentioned above, the analysis framework used by the authors is the 
so-called \itshape boundedness by entropy method\normalfont. The main idea of this technique is to write the above system of equations as a formal 
gradient flow and derive estimates on the solutions $(u_0, \cdots, u_n)$ using the decay of some 
well-chosen entropy functional. We present in Section~\ref{sec:gradientflow} the formal gradient flow structure of (\ref{eq:system_part}) and 
recall the results of~\cite{Jungel} in Section~\ref{sec:Jungel}.

\medskip

\begin{remark}
This model is linked to the so-called \normalfont Stefan-Maxwell \itshape model, studied in~\cite{JungelStelzer1,Boudin}. Indeed, the model considered in the 
latter paper reads
\begin{equation}\label{eq:system_SM}
\left\{
\begin{array}{ll}
\partial_t u - \dive_x \left( A(u)^{-1} \nabla_x u \right) = 0, &  \quad \mbox{ for } (t,x) \in (0,T] \times \Omega,\\
(A(u) \nabla_x u)\cdot \textbf{n} = 0, & \quad \mbox{ for }(t,x)\in (0,T] \times \partial \Omega,\\
u(0,x) = u^0(x), &\quad \mbox{ for } x \in \Omega,\\
\end{array}
\right .
\end{equation}
where $A$ is defined by (\ref{eq:defA}).
\end{remark}

%%%%%%%%%%%%%%%%%%%%%%%%%%%%%%%%%%%%%%%%%%%%%%%%%
\subsubsection{Formal gradient flow structure of (\ref{eq:system_part})}\label{sec:gradientflow}

We detail in this section the formal gradient flow structure of the system~(\ref{eq:system_part}). 

Let $\cD\subset \R^n$ be defined by
\begin{equation}\label{eq:defD}
\cD := \left \{  (u_1, \cdots,u_n) \in (\R^*_+)^n, \quad \sum \limits_{i=1}^n u_i <1\right \} \subset (0,1)^n.
\end{equation}
Let us introduce the classical \itshape entropy density \normalfont $h$ (see for instance \cite{Burger}, \cite{Jungel}, \cite{JungelZamponi2} and \cite{Mielke}) 
\begin{equation}\label{eq:defh}
h : \left\{
\begin{array}{ccc}
 \overline{\cD}  & \longrightarrow & \R \\
u:=(u_i)_{1\leq i \leq n} & \longmapsto &h(u) =   \sum \limits_{i=1}^n u_i \log u_i + (1-\rho_u)\log (1-\rho_u), \\
\end{array}
\right .
\end{equation}
where $\rho_u:= \sum_{i=1}^n u_i$. Some properties of $h$ can be easily checked: 
\begin{itemize}
 \item[(P1)] the function $h$ belongs to $ \cC^0( \overline{\cD}) \cap \cC^2(\cD)$; consequently, $h$ is bounded on $\overline{\cD}$;
 \item[(P2)] the function $h$ is strictly convex on $\cD$; 
 \item[(P3)] its derivative 
$$
Dh : \left\{
\begin{array}{ccc}
\cD & \longrightarrow & \R^n \\
(u_i)_{1\leq i \leq n} & \mapsto & \left( \log \frac{u_i}{1-\rho_u}\right)_{1\leq i \leq n},\\
\end{array}
\right .
$$ is invertible and its inverse is given by
$$
(Dh)^{-1} : \left\{
\begin{array}{ccc}
 \R^n & \longrightarrow & \cD\\
 (w_i)_{1\leq i \leq n} & \mapsto & \frac{e^{w_i}}{1 + \sum_{j=1}^n e^{w_j}}.\\
\end{array}
\right .
$$
\end{itemize}

In the following, we denote by $D^2h$ the Hessian of $h$. The \itshape entropy functional \normalfont $\mathcal{E}$ is defined by 
\begin{equation}
\label{eq:defE}
\cE : \left\{
\begin{array}{ccc}
 L^\infty(\Omega; \overline{\cD})  & \longrightarrow & \R \\
u & \longmapsto &\cE (u): =  \int_{\Omega} h(u(x))\,dx. \\
\end{array}
\right .
\end{equation}
 {Throughout the article, for all $u\in L^\infty(\Omega; \cD)$, we shall denote by $D\mathcal{E}(u)$ the measurable vector-valued function defined by
$$
D\mathcal{E}(u): \left\{ 
\begin{array}{ccc}
 \Omega & \to & \R^n \\
 x & \mapsto & Dh(u(x)).\\
\end{array}
\right .
$$
}

\medskip 

The system \eqref{eq:system_part} can then be formally rewritten under the following gradient flow structure 

 \begin{equation} \label{eq:System_GF}
 \left \{ 
 \begin{array}{ll}
 \partial_t u - \dive_x \left( M(u) \nabla_x D\cE(u) \right) = 0, & \quad  \mbox{ for } (t,x)\in \R_+^*\times \Omega,\\
 \left( M(u) \nabla_x D\cE(u) \right) \cdot \textbf{n} = 0, & \quad \mbox{ for } (t,x) \in \R_+^*\times \partial \Omega, \\
 u(0,x) = u^0(x), & \quad \mbox{ for } x \in \Omega,\\
 \end{array} 
 \right. 
\end{equation}  
where $M:\overline{\cD} \to \R^{n\times n}$ is the so-called \textit{mobility matrix} of the system defined for all $u\in \cD$ by 
$$
M(u):= A(u) (D^2h(u))^{-1}. 
$$
More precisely, it holds that for all $u\in \overline{\cD}$, $M(u) = (M_{ij}(u))_{1\leq i,j \leq n}$ where for all $1\leq i\neq j \leq n$, 
\begin{equation}\label{eq:defM}
M_{ii}(u) = K_{i0} (1-\rho_u) u_i + \sum_{1\leq j\neq i \leq n} K_{ij} u_i u_j \quad \mbox{ and } \quad M_{ij}(u) = -K_{ij}u_i u_j.
\end{equation}

\subsection{Existence of global weak solutions by the boundedness by entropy technique}\label{sec:Jungel}

The formal gradient flow formulation of a system of cross-diffusion equations 
is a key point in the boundedness by entropy technique. In the example presented in Section~\ref{sec:example}, it implies in particular that $\mathcal{E}$ 
is a Lyapunov functional for the system \eqref{eq:system_part}~\cite{Burger,Jungel}. 
However, the mobility matrix obtained for these systems is not a concave function of the densities, so that standard gradient flow theory arguments (such as the minimizing movement method) 
cannot be applied in this context~\cite{MatthesZinsl, Dolbeault, JKO, Mielke}. However, the existence of a global weak solution to~\eqref{eq:system_part} can still be proved.
Let us recall here a simplified version of Theorem~2 of~\cite{Jungel} which is adapted to our context.
 
\begin{theorem}[Theorem~2 of~\cite{Jungel}] \label{th:Jungel}
Let $\cD\subset \R^n$ be the domain defined by (\ref{eq:defD}). Let $A: u\in \overline{\cD} \mapsto A(u):=(A_{ij}(u))_{1\leq i,j\leq n} \in \R^{n\times n}$ be 
 a matrix-valued functional defined on $\overline{\cD}$ satisfying $A\in \cC^0(\overline{\cD}; \R^{n\times n})$ and the following assumptions: 
 \begin{itemize}
  \item [(H1)] There exists a bounded from below convex function $h\in \cC^2(\cD, \R)$ such that its derivative $Dh:\cD \to \R^n$ is invertible on $\R^n$; 
  \item[(H2)]  {There exists $\alpha >0$, and for all $1\leq i \leq n$, there exist $1\geq m_i>0,$ such that for all $z = (z_1, \cdots, z_n)^T \in \R^n$ and $u=(u_1,\cdots, u_n)^T\in \cD$, 
  $$
  z^T D^2h(u)A(u) z \geq \alpha \sum_{i=1}^n u_i^{2m_i-2} z_i^2.
  $$}
 \end{itemize}
Let $u^0 \in L^1(\Omega; \cD)$ so that $w^0:= Dh(u^0) \in L^\infty(\Omega; \R^n)$. Then, there exists a weak solution $u$ with initial condition $u^0$ to 
\begin{equation}\label{eq:Jungeltrue}
\left\{
\begin{array}{ll}
\partial_t u = \mbox{\rm div}_x(A(u)\nabla_x u), & \quad \mbox{ for } (t,x)\in \R_+^*\times \Omega,\\
\left( A(u)\nabla_x u \right) \cdot \textbf{n} = 0, & \quad \mbox{ for } (t,x) \in \R_+^*\times \partial\Omega,\\
\end{array}
\right .
\end{equation}
such that for almost all $(t,x)\in \R_+^*\times \Omega$, $u(t,x)\in \overline{\cD}$ with 
$$
u\in L^2_{\rm loc}(\R_+; H^1(\Omega,\R^n))\mbox{ and } \partial_t u \in L^2_{\rm loc}(\R_+; (H^1(\Omega;\R^n))').
$$
%Besides, for all $1\leq i \leq n$, $|u_i|^{m_i} \in L^2((0,T]; H^1(\Omega))$.
\end{theorem}

Lemma~\ref{lem:cor} states that the prototypical example presented in Section~\ref{sec:example} falls into the framework of Theorem~\ref{th:Jungel}. The proof of the latter
is given Section~\ref{sec:proof_cor} for the sake of completeness, and relies on ideas introduced in~\cite{JungelZamponi2}. 
\begin{lem}\label{lem:cor}
Let $\cD\subset \R^n$ be the domain defined by (\ref{eq:defD}) and $A: u\in \overline{\cD} \mapsto A(u):=(A_{ij}(u))_{1\leq i,j\leq n} \in \R^{n\times n}$ be the matrix-valued function defined by (\ref{eq:defA}). 
 Then, $A \in \cC^0(\overline{\cD}; \R^{n\times n})$ and satisfies assumptions (H1)-(H2) of Theorem~\ref{th:Jungel}, with $h$ given by (\ref{eq:defh}),  {$\alpha = \min_{1\leq i \neq j \leq n} K_{ij}$} and
 $m_i = \frac{1}{2}$ for all $1\leq i \leq n$. 
\end{lem}

The existence of global weak solutions to (\ref{eq:system_part}) is then a direct consequence of Theorem~\ref{th:Jungel} and Lemma~\ref{lem:cor}.

\medskip

Let us point out that the uniqueness of solutions to general systems of the form (\ref{eq:Jungeltrue}) 
remains an open theoretical question, at least up to our knowledge. It can be obtained in some particular cases. When the diffusion matrix $A$ is defined by (\ref{eq:defA}) and 
when all the diffusion coefficients $K_{ij}$ are identically equal 
to some constant $K>0$, the uniqueness of the solution can be trivially obtained since the system boils down to a set of $n$ 
decoupled heat equation for the evolution of the density of each species. 

% \medskip
% 
% Let us make a final remark. The existence of a global weak solution for system~(\ref{eq:system_part}) is thus proved under one of the two following assumptions on the values 
% of the cross-diffusion coefficients $(K_{ij})_{0\leq i \neq  j \leq n}$:
% \begin{itemize}
%  \item either $K_{ij} = K_{ji} >0$ for all $0\leq i\neq j \leq n$; 
%  \item or $K_{i0} = K_{0i} >0$ for all $1\leq i \leq n$ and $K_{ij} = K_{ji} = 0$ for all $1\leq i\neq j \leq n$.
% \end{itemize}
% However, the existence of a global weak solution remains an open problem as soon as the coefficients do not satisfy one of these conditions. 

%%%%%%%%%%%%%%%%%%%%%%%%%%%%%%%%%%%%%%%%%%%%%%%%%
\section{Case of non-zero flux boundary conditions and moving domain}\label{sec:flux}

In the sequel, we restrict the study to the case when $d=1$. In this section, we propose a model for the description of a PVD process and present theoretical results whose proofs 
are postponed to Section~\ref{sec:proofs}. The global existence of a weak solution is proved. 
The long-time behaviour of such a solution is studied in the case of constant external fluxes. Lastly, under the assumption 
that the coefficients $K_{ij}$ are chosen so that there is a unique solution to the system, we prove the existence of a solution to an optimization problem.

\subsection{Presentation of the model}\label{sec:modelflux}

For the sake of simplicity, we assume that non-zero fluxes are only imposed on the right-hand side of 
the domain occupied by the solid. At some time $t>0$, this domain is denoted by $\Omega_t := (0, e(t))$ where $e(t)>0$ models the thickness of the layer. 
Initially, we assume that the domain $\Omega_0$ occupied by the solid at time $t=0$ is the interval $(0,e_0)$ for some initial thickness $e_0 > 0$.

\medskip

The evolution of the thickness of the film $e(t)$ is determined by the external fluxes of the atomic species 
that are absorbed at its surface. 
More precisely, let us assume that there are $n+1$ different chemical species composing the solid layer and let $(\phi_0, \cdots, \phi_n)$ belong to $L^{\infty}_{\rm loc}(\R_+; \R_+^{n+1})$. For all $0\leq i \leq n$, the function $\phi_i(t)$
represents the flux of the species $i$ absorbed at the surface at time $t>0$ and is assumed to be non-negative. 
In this one-dimensional model, the evolution of the thickness of the solid is assumed to be given by
\begin{equation}\label{eq:evole1}
e(t):= e_0 + \int_0^t \sum_{i=0}^n \phi_i(s)\,ds.
\end{equation}
In the following, we will denote by $\varphi:= (\phi_1, \cdots , \phi_n)^T$ (see Figure~\ref{fig:fig1}).

\medskip

%\begin{figure}\label{fig:fig1}
%\centering
%\input{./PVD.pstex_t}
%\caption{ {Illustration of the composition of the film layer at time $t$ in the case $n=2$}}
%\label{fig:fig1}
%\end{figure}

 \begin{figure}
 \centering
 \includegraphics[width=10cm, height=6cm]{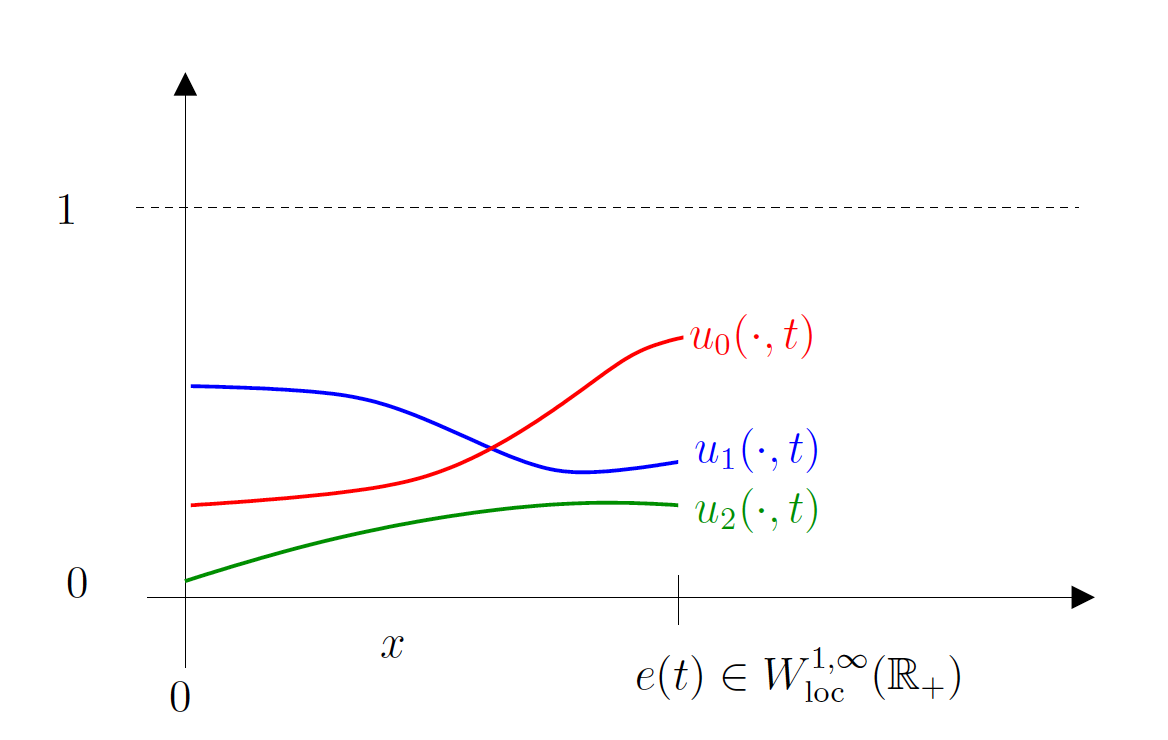}
 \caption{ {Illustration of the composition of the film layer at time $t$ in the case $n=2$}}
 \label{fig:fig1}
 \end{figure}

\medskip

For all $t\geq 0$ and $0\leq i \leq n$, the local concentration of species $i$ at time $t$ and point $x\in (0,e(t))$ is denoted by $u_i(t,x)$. The evolution of the vector $u:=(u_1,\cdots,u_n)$ 
is given by the system of cross-diffusion equations
\begin{equation}\label{eq:cross}
\partial_t u - \partial_x\left( A(u) \partial_x u\right) = 0, \mbox{ for } t\in \R_+^*, \; x\in (0,e(t)), 
\end{equation}
where $A: \overline{\cD}\to \R^{n\times n}$ is a well-chosen diffusion matrix satisfying (H1)-(H2). 

\medskip

We consider that for every $t>0$, the system satisfies the following conditions on the boundary $\partial \Omega_t$:
\begin{equation} \label{eq:bord}
\left( A(u) \partial_x u \right)(t, 0) = 0 \mbox{ and } \left( A(u) \partial_x u \right)(t, e(t)) + e'(t) u(t,e(t)) = \varphi(t). 
\end{equation}   

An easy calculation shows that these boundary conditions, in addition to~\eqref{eq:evole1} and~\eqref{eq:cross}, ensure that, for all $0 \leq i \leq n$, 
$$
\frac{d}{dt}\left( \int_{\Omega_t} u_i(t,x)\,dx\right) = \phi_i(t).
$$
Indeed, it holds that
\begin{align*}
\frac{d}{dt}\left( \int_{\Omega_t} u(t,x)\,dx\right)  &=  \int_{0}^{e(t)} \partial_t u(t,x) \,dx + e'(t)u(t,e(t)),\\
 & = \int_{0}^{e(t)} \partial_x\left( A(u)\partial_x u\right) + e'(t)u(t,e(t)),\\
  & = (A(u)\partial_x u)(t,e(t))  + e'(t)u(t,e(t)) - (A(u)\partial_x u)(t,0),\\ 
 & =  \varphi(t).\\
\end{align*}

 {The calculation for the $0^{th}$ species reads:
\begin{align*}
 \frac{d}{dt}\left( \int_{\Omega_t} u_0(t,x)\,dx\right) & = \frac{d}{dt}\left( |\Omega_t| - \sum_{i=1}^n \int_{\Omega_t} u_i(t,x)\,dx \right)\\
 & = e'(t) - \sum_{i=1}^n \frac{d}{dt}\left( \int_{\Omega_t} u_i(t,x)\,dx \right)\\
 & = \sum_{i=0}^n \phi_i(t) - \sum_{i=1}^n \phi_i(t) = \phi_0(t).\\
\end{align*}
}

% 
% %%%%%%%%%%%%%%%%%%%%%%%%%%%%%%%%%%%%%%%%%%%%%%%%%

\medskip 

To sum up, the final system of interest reads:
 \begin{equation} \label{eq:System_1D}
 \left \{ 
 \begin{array}{ll}
 e(t) = e_0 + \int_{0}^t \sum_{i=0}^n \phi_i(s) \,ds, & \quad \mbox{ for } t \in \R_+^*,\\
 \partial_t u - \partial_x \left( A(u) \partial_x u\right) = 0, & \quad \mbox{ for } t\in \R_+^*, \; x\in (0,e(t)),\\
 \left( A(u) \partial_x u \right)(t,0) = 0, & \quad \mbox{ for } t \in \R_+^*,\\
  \left( A(u) \partial_x u \right)(t,e(t)) + e'(t) u(t,e(t)) = \varphi(t), & \quad \mbox{ for } t \in \R_+^*,\\
 u(0,x) = u^0(x), &\quad \mbox{ for } x \in (0,e_0),
 \end{array} 
 \right. 
\end{equation} 
where $u^0\in L^1(0,e_0)$ is an initial condition satisfying $u^0(x)\in \cD$ for almost all $x\in (0, e_0)$. We assume in addition that $w^0:= Dh(u^0)$ belongs to $L^\infty((0,e^0); \R^n)$.

\subsubsection{Rescaled version of the model}\label{sec:modelrescaled}

We introduce here a rescaled version of system~\eqref{eq:System_1D}. For all $0\leq i \leq n$, $t\geq 0$ and $y\in (0,1)$, let us denote by $v_i(t,y):= u_i(t, e(t)y)$. It holds that
$$
\partial_t v(t,y) = \partial_t u(t,e(t)y) + e'(t) y \partial_x u(t, e(t)y) \; \mbox{ and } \; \partial_y v(t,y) = e(t) \partial_x u(t, e(t)y),
$$
where $v:=(v_1, \cdots, v_n)$.
Thus, $u$ is a solution of~\eqref{eq:System_1D} if and only if $v$ is a solution to the following system:
\begin{equation}\label{eq:rescaled}
\left\{
\begin{array}{ll}
\dps e(t) = e_0 + \int_0^t \sum_{i=0}^n \phi_i(s)\,ds, & \mbox{ for } t \in \R_+^*,\\
\partial_t v - \frac{1}{e(t)^2}\partial_y\left( A(v) \partial_y v\right) - \frac{e'(t)}{e(t)}y \partial_y v = 0, & \mbox{ for }(t,y)\in \R_+^* \times (0,1),\\ 
\frac{1}{e(t)}(A(v) \partial_yv)(t,1) + e'(t) v(t,1) = \varphi(t), & \mbox{ for } (t,y)\in \R_+^*\times (0,1),\\
\frac{1}{e(t)}(A(v)\partial_y v)(t,0) = 0, & \mbox{ for }(t,y)\in\R_+^* \times (0,1)\\
v(0,y) = v^0(y), & \mbox{ for }y \in (0,1),\\
\end{array}
\right .
\end{equation}
where $v^0(y):= u^0(e_0 y)$. 

Proving the existence of a global weak solution to (\ref{eq:System_1D}) is equivalent to proving the existence of a global weak solution to (\ref{eq:rescaled}). 

\medskip

 {
Actually, it can be seen that the entropy of the system (\ref{eq:rescaled}) satisfies a formal inequality at the continuous level which is at the heart of the proof of our existence result. Indeed, let us denote by 
$$
\mathcal{E}(t):= \int_0^1 h(v(t,y))\,dy,
$$
where $v$ is a solution to (\ref{eq:rescaled}). Then, formal calculations yield that
\begin{align*}
\frac{d\mathcal{E}}{dt}(t) & = \int_0^1 \partial_t v(t,y) \cdot Dh(v(t,y))\,dy\\
& = \frac{1}{e(t)^2} \int_{0}^1 \partial_y\left( A(v(t,y)) \partial_y v(t,y)\right) \cdot Dh(v(t,y))\,dy + \frac{e'(t)}{e(t)} \int_0^1 y \partial_y v(t,y) \cdot Dh(v(t,y))\,dy\\
& = -\frac{1}{e(t)^2} \int_{0}^1  \partial_y v(t,y) \cdot D^2h(v(t,y)) A(v(t,y)) \partial_y v(t,y)\,dy \\
& + \frac{1}{e(t)^2} \left( A(v(t,1))\partial_y v(t,1)\right) \cdot Dh(v(t,1))  + \frac{e'(t)}{e(t)}\int_0^1 y \partial_y (h(v(t,y)))\,dy\\
& = -\frac{1}{e(t)^2}  \int_{0}^1 \partial_y v(t,y) \cdot D^2h(v(t,y)) A(v(t,y)) \partial_y v(t,y)\,dy + \frac{1}{e(t)}\left( \varphi(t) - e'(t)v(t,1)\right) \cdot  Dh(v(t,1))\\
& + \frac{e'(t)}{e(t)}h(v(t,1)) - \frac{e'(t)}{e(t)} \int_0^1  h(v(t,y))\,dy.\\
\end{align*}
Denoting by $\overline{f}(t):= \frac{\varphi(t)}{e'(t)}$, it holds that $\overline{f}(t)\in \overline{\cD}$ for all $t>0$. Besides, using assumption (H2), we obtain that
$$
- \int_0^1 \partial_y v(t,y) \cdot D^2h(v(t,y)) A(v(t,y)) \partial_y v(t,y)\,dy \leq 0,
$$
which yields that
\begin{align*}
\frac{d\mathcal{E}}{dt}(t) & \leq \frac{e'(t)}{e(t)} \left[ h(v(t,1) + Dh(v(t,1)) \cdot \left( \overline{f}(t) - v(t,1) \right) - \int_0^1 h(v(t,y))\,dy \right].
\end{align*}
Using the convexity of $h$, we obtain that $h(v(t,1) + Dh(v(t,1)) \cdot \left( \overline{f}(t) - v(t,1) \right) \leq h(\overline{f}(t))$, so that
\begin{equation}\label{eq:formalineq}
\frac{d\mathcal{E}}{dt}(t) \leq \frac{e'(t)}{e(t)}\left[ h(\overline{f}(t))  - \int_0^1 h(v(t,y))\,dy \right].
\end{equation}
Inequality (\ref{eq:formalineq}) is not an entropy dissipation inequality in the sense that the quantity $\mathcal{E}(t)$ may increase with time. However, 
using the fact $e' \in L^\infty_{\rm loc}(\R_+;\R_+)$ and assumption (H3), it implies that the quantity $\mathcal{E}(t)$ cannot blow up in finite time, which is sufficient for our purpose. 
}

\subsection{Theoretical results}

\subsubsection{Global in time existence of weak solutions}\label{sec:existence}

Our first result deals with the global in time existence of bounded weak solutions to \eqref{eq:rescaled} (and thus to (\ref{eq:System_1D})). 
\begin{theorem}\label{th:moving}
Let $\cD:=\{(u_1, \cdots, u_n)^T \in (\R_+^*)^n, \; \sum_{i=1}^n u_i < 1\} \subset (0,1)^n$. Let $A: \overline{\cD}\to \R^{n\times n}$ 
be a matrix-valued functional satisfying $A\in \cC^0(\overline{\cD}; \R^{n\times n})$ and assumptions (H1)-(H2) of Theorem~\ref{th:Jungel} for some well-chosen 
entropy density $h: \overline{\cD} \to \R$. We assume in addition that 
 \begin{itemize}
  \item[(H3)] $h\in \cC^0(\overline{\cD})$.
 \end{itemize}
Let $e_0>0$, $u^0\in L^1((0,e_0); \cD)$ so that $w^0:= (Dh)^{-1}(u^0) \in L^\infty((0,e_0); \R^n)$ and $(\phi_0, \cdots , \phi_n)\in L^{\infty}_{\rm loc}(\R_+; \R_+^{n+1})$. Let us define for almost all $y\in (0,1)$, 
$v^0(y):=u^0(e_0 y)$ and $\varphi:=(\phi_1, \cdots, \phi_n)^T$. Then, there exists a weak solution $v$ with initial condition $v^0$ to~\eqref{eq:rescaled} 
such that for almost all $(t,y)\in \R_+^*\times (0,1)$, $v(t,y)\in \overline{\cD}$. Besides,
$$
v\in L^2_{\rm loc}(\R_+; H^1((0,1);\R^n)) \mbox{ and }\partial_t v \in L^2_{\rm loc}(\R_+; (H^1((0,1);\R^n))').
$$
In particular, $v \in \cC^0(\R_+; L^2((0,1); \R^n))$.
\end{theorem}

Let us point out that the example described in Section~\ref{sec:example} satisfies all the assumptions 
of Theorem~\ref{th:moving} since the entropy density $h$ defined by (\ref{eq:defh}) belongs to $\cC^0(\overline{\cD})$. Let us also point here that the form of (\ref{eq:rescaled}) is different 
from the system considered in~\cite{Jungel} through i) the boundary conditions and ii) the existence of the drift term $\frac{e'(t)}{e(t)}y \partial_y v$. 

The strategy of proof developped in~\cite{Burger,Jungel} is still adapted to our case though, because a discrete entropy inequality can still be obtained. 
The proof of Theorem~\ref{th:moving} is given in full details in Section~\ref{sec:proof_exists}.

\subsubsection{Long-time behaviour for constant fluxes}\label{sec:longtime}

In the case when the fluxes are constant in time, we obtain long-time asymptotics for the functions $v_i$, provided that 
the entropy density $h$ is given by (\ref{eq:defh}). More precisely, the following result holds:
\begin{prop}\label{prop:longtime}
Under the assumptions of Theorem~\ref{th:moving}, let us make the following additional hypotheses:
\begin{itemize}
 \item[(T1)] for all $0\leq i \leq n$, there exists $\overline{\phi}_i > 0$ so that $\phi_i(t) = \overline{\phi}_i$, for all $t\in \R_+$; 
 \item[(T2)] for all $u\in \overline{D}$, the entropy density $h$ can be chosen so that $h(u) = \sum_{i=1}^n u_i \log u_i + (1-\rho_u) \log (1-\rho_u)$.
\end{itemize}
For all $0\leq i \leq n$, let us define  {$\overline{f}_i:= \frac{\overline{\phi}_i}{\sum_{j=0}^n \overline{\phi}_j}$ and by $\overline{f}:=(\overline{f}_i)_{1\leq i \leq n}\in \cD$. Let us also denote by
$$
\overline{h}: \left\{ \begin{array}{ccc}
                       \overline{\cD} & \mapsto & \R\\
u & \mapsto & \sum_{i=1}^n u_i \ln \left( \frac{u_i}{\overline{f}_i}\right) + \left(1-\rho_u\right) \ln \left( \frac{1-\rho_u}{1-\rho_{\overline{f}}}\right)\\
                      \end{array}\right.
$$
the relative entropy associated to $h$ and $\overline{f}$. Then, there exists a global weak solution $v$ to (\ref{eq:rescaled}) and a 
constant $C>0$ such that
\begin{equation}\label{eq:relentropy}
\int_0^1 \overline{h}\left( v(t,y)\right)\,dy \leq \frac{C}{t+1},
\end{equation}
and
\begin{equation}\label{eq:L1}
\forall 1\leq i \leq n, \quad \|v_i(t,\cdot) - \overline{f}_i\|_{L^1(0,1)} \leq \frac{C}{\sqrt{t+1}} \mbox{ and } \left\|\left( 1 - \rho_{v(t,\cdot)}\right) - \overline{f}_0\right\|_{L^1(0,1)} \leq \frac{C}{\sqrt{t+1}}.
\end{equation}
}
\end{prop}

The proof of Proposition~\ref{prop:longtime} is given in Section~\ref{sec:proof_time}. Numerical results presented in Section~\ref{sec:numerics} illustrate the rate of convergence of 
the rescaled concentrations to constant profiles in $\mathcal{O}\left( \frac{1}{t} \right)$.

\medskip

 {Let us comment here on assumption (T2). Actually, in the proof, we use the following properties of the logarithmic entropy density: 
\begin{itemize}
 \item $\overline{h}\geq 0$ and if $u\in\overline{\cD}$ satisfies $\overline{h}(u) = 0$, then necessarily $u = \overline{f}$; 
 \item There exists a constant vector $g\in \R^n$ such that for all $u\in \cD$, $D\overline{h}(u) = Dh(u) + g$. 
 \item The relative entropy density $\overline{h}$ satisfies a Csiz\`ar-Kullback type inequality. 
\end{itemize}
Similar long-time asymptotics results can be obtained for general entropy densities satisfying these three properties. For the sake of simplicity, we chose to restrict ourselves to the case of logarithmic entropy density in 
Proposition~\ref{prop:longtime}.}

\medskip

 {The central ingredient of the proof is the following formal entropy inequality. In the case when $h$ is given by (\ref{eq:defh}), it can be easily seen that $\overline{h}$ is also a valid entropy density for the diffusion 
coefficient $A$ in the sense that $\overline{h}$ also satisfies assumptions (H1)-(H2)-(H3). Thus, inequality (\ref{eq:formalineq}) holds with $\overline{h}$ instead of $h$ so that
$$
\frac{d \overline{\mathcal{E}}}{dt}(t) \leq \frac{e'(t)}{e(t)}\left[ \overline{h}(\overline{f})  - \int_0^1 \overline{h}(v(t,y))\,dy \right]= \frac{e'(t)}{e(t)}\left[ \overline{h}(\overline{f})  -\overline{\mathcal{E}}(t) \right],
$$
where for all $t>0$,  $\overline{\mathcal{E}}(t):= \int_0^1 \overline{h}(v(t,y))\,dy$. 
Denoting by $V:= \sum_{i=0}^n\overline{\phi}_i$, it holds that $e'(t) = V$ and $e(t) = e_0 + Vt$ for all $t\geq 0$. Finally, using the fact that $\overline{h}\geq 0$ and that $\overline{h}(\overline{f})= 0$, we obtain that
$$
\left(\frac{e_0}{V} + t\right) \frac{d  \overline{\mathcal{E}}}{dt}(t)+ \overline{\mathcal{E}}(t)  = \frac{d}{dt}\left( \left(\frac{e_0}{V} + t\right) \overline{\mathcal{E}}(t) \right) \leq 0.
$$
This inequality implies that there exists a constant $C>0$ such that for all $t\geq 0$, 
$$
\overline{\mathcal{E}}(t) \leq \frac{C}{t+1}.
$$
The rates on the $L^1$ norm of the solutions are then obtained using the Csiz\`ar-Kullback inequality. 
}

\subsubsection{Optimization of the fluxes}\label{sec:control}

As mentioned in the introduction, our main motivation for studying this system is the control of the gazeous fluxes injected during a PVD process. 
 {It is assumed here that the wafer remains in the hot chamber where the different atomic species are injected during a time $T>0$. The cross-diffusion phenomena occur in the bulk of the 
thin film layer because of the high temperatures that are imposed during the process. Once the wafer is taken out of the chamber, the composition of the film is \itshape freezed \normalfont and 
no diffusion phenomena take place anymore. The profiles of the local volumic fractions of the different chemical species in the film thus remain unchanged after the time $T$. It is of practical interest to 
adapt the fluxes through time so that these final concentration profiles are as close as possible to target functions chosen a priori.} 

\medskip

Let $e_0>0$ be the initial thickness of the solid. In practice, the maximal value of the fluxes which can be injected is limited due to device 
constraints. Let $F>0$ and let us then denote by $\Xi := \left\{ \Phi \in L^\infty((0,T); \R_+^{n+1}), \quad \|\Phi\|_{L^\infty} \leq F\right\}$. 
For all $\Phi:=(\phi_0, \cdots, \phi_n)\in \Xi$, we denote by $e_\Phi: t\in [0,T] \mapsto e_0 + \int_0^t \sum_{i=0}^n \phi_i(s)\,ds$ the time-dependent thickness of the film, and by 
$v_\Phi$ a solution to (\ref{eq:rescaled}) associated with the external fluxes $\Phi$. 

\medskip

Let us point out here the uniqueness of a solution to (\ref{eq:System_1D}) (or (\ref{eq:rescaled})) remains an open problem in general.  {When the diffusion matrix $A$ is defined by (\ref{eq:defA}), the only case for which uniqueness of a global solution can be obtained is the trivial case where the cross-diffusion coefficients 
$K_{ij}$ are identical to some constant $K>0$ for all $0\leq i \neq j \leq n$. Indeed, in this case, it can be seen that the system (\ref{eq:rescaled}) can be written as a set of $n$ independent advection-diffusion PDEs 
on each of the rescaled concentration profiles $v_i$ ($1\leq i \leq n$). Thus, we will have to make some assumption on the cross-diffusion coefficients $(K_{ij})_{0\leq i \neq j \leq n}$ in the general case. }

\medskip

We make the following assumption on the diffusion matrix $A$:
\begin{itemize}
 \item[(C1)] For any $\Phi \in \Xi$, there exists a unique global weak solution $v_\Phi$ to system~\eqref{eq:rescaled} so that for almost all $(t,y)\in \R_+^* \times (0,1)$, $v_\Phi(t,y) \in \overline{\cD}$.
\end{itemize}

The goal of the optimization problem consists in the identification of optimal time-dependent non-negative functions 
$\Phi \in \Xi$ so that the final thickness of the film $e_\Phi(T)$ and the (rescaled) concentration profiles for the different chemical species 
$v_\Phi(T,\cdot)$ at the end of the fabrication process are as close as possible to desired targets denoted by 
$e_{\rm opt} > e_0$ and $v_{\rm opt} \in L^2((0,1); \overline{\cD})$.

The real-valued functional $\cJ: \Xi \to \R$ defined by
\begin{equation}
\forall \Phi \in \Xi, \quad \cJ(\Phi):= |e_\Phi(T) - e_{\rm opt}|^2 + \|v_\Phi(T,\cdot) - v_{\rm opt}\|^2_{L^2(0,1)},
\label{eq:cost} 
\end{equation} 
is the cost function we consider here. 
More precisely, we have the following result, which is proved in Section~\ref{sec:proof_optimal}. 
\begin{prop}\label{prop:optimal}
Under the assumptions of Theorem~\ref{th:moving}, let us make the additional assumption (C1). Then, the functional $\mathcal{J}$ is well-defined and there exists 
a minimizer $\Phi^* \in \Xi$ to the minimization problem
\begin{equation}
\label{eq:optimal}
\Phi^* \in \mathop{\mbox{argmin}}_{\Phi \in \Xi} \cJ(\Phi). 
\end{equation}
\end{prop}
Of course, uniqueness of such a solution $\Phi^*$ is not expected in general.

 %%%%%%%%%%%%%%%%%%%%%%%%%%%%%%%%%%%%%%%%%%%%%
\section{Proofs} \label{sec:proofs}
\subsection{Proof of Lemma~\ref{lem:cor}}\label{sec:proof_cor}
Let us prove that the matrix-valued function $A$ defined in (\ref{eq:defA}) satisfies the assumptions of Theorem~\ref{th:Jungel} with the entropy functional $h$ given by (\ref{eq:defh}). 

\medskip

As mentioned in Section~\ref{sec:example}, the entropy density $h$ belongs to $\cC^0(\overline{\cD}; \R) \cap \cC^2(\cD;\R)$ (thus is bounded on $\overline{\cD}$), is strictly convex on $\cD$, and its derivative $Dh: \cD \to \R^n$ is invertible. 
As a consequence, $h$ satisfies assumption (H1) of Theorem~\ref{th:Jungel}.  

\medskip

Let us now prove that assumption (H2) of Theorem~\ref{th:Jungel} is satisfied with $m_i = \frac{1}{2}$ for all $1\leq i \leq n$. 
To this aim, we follow the same strategy of proof as the one used in~\cite{JungelZamponi2}. 
Let us prove that there exists  {$\beta >0$ such that for all $u\in \cD$, 
\begin{equation}\label{eq:ineg}
H(u)A(u) \geq \beta \Lambda(u),
\end{equation}
$$  \; \mbox{ where } H(u):= D^2h(u), \; \Lambda(u):= \mbox{diag}\left( \left(\frac{1}{u_i}\right)_{1\leq i \leq n}\right) \mbox{ and } \beta:= \min_{0\leq i\neq j \leq n} K_{ij}. $$
}

 {
This inequality implies (H2) with $\alpha = \beta$ and $m_i = \frac{1}{2}$ for all $1\leq i \leq n$. }

Let $u\in D$. We have for all $1\leq i,j \leq n$,  
$$
H_{ii}(u) = \frac{1}{u_i} + \frac{1}{1-\rho_u} \mbox{ and } H_{ij}(u)= \frac{1}{1-\rho_u} \mbox{ if } i \neq j. 
$$
Introducing $P(u) := (P_{ij}(u))_{1\leq i,j \leq n}$, where for all $1\leq i ,j \leq n$,  
$$
P_{ii}(u) = 1-u_i \mbox{ and } P_{ij}(u) = -u_i \mbox{ if } i \neq j, 
$$
it holds that $H(u) P(u) = \Lambda(u)$.  {Thus, $H(u) A(u) - \beta \Lambda(u) = H(u)(A(u) - \beta P(u))$. It can be easily checked that 
$A(u) - \beta P(u) = \widetilde{A}(u) + \beta D(u)$, where $\widetilde{A}(u)$ has the same structure as $A(u)$ but with diffusion coefficients $K_{ij}-\beta$ instead of $K_{ij}$, and 
$D(u) := (D_{ij}(u))_{1\leq i,j \leq n}$ where $D_{ij}(u) = u_i$ for all $1\leq i \leq n$.}

On the one hand, $H(u)D(u) = \frac{1}{1-\rho_u}Z$ where $Z$ is the $n\times n$ matrix whose all coefficients are identically equal to $1$.
Since the matrix $Z$ is a semi-definite positive matrix, so is $H(u)D(u)$. 

 {On the other hand, since $h$ is strictly convex on $\cD$, $H(u) \widetilde{A}(u)$ is semi-definite positive if and only if $\widetilde{M}(u):= \widetilde{A}(u) H(u)^{-1}$ is semi-definite positive. Indeed, for all $z \in \R^n$, 
we have $z^T H(u)\widetilde{A}(u) z = (H(u) z)^T \left(\widetilde{A}(u) H(u)^{-1}\right) (H(u) z)$. 
It can be observed that
$\widetilde{M}(u) = (\widetilde{M}_{ij}(u))_{1\leq i,j \leq n}$, where for all $1\leq i,j \leq n$,
$$
\widetilde{M}_{ii}(u) = (K_{i0}-\beta) (1-\rho_u) u_i + \sum_{1\leq j \neq i \leq n} (K_{ij}-\beta) u_i u_j \mbox{ and } \widetilde{M}_{ij}(u) = - (K_{ij}-\beta) u_i u_j \mbox{ if } j\neq i.
$$
For all $z = (z_1, \cdots, z_n)^T \in \R^n$, we have
\begin{align*}
 z^T \widetilde{M}(u) z  & = \sum_{i=1}^n (K_{i0}-\beta) (1-\rho_u) u_i z_i^2 + \sum_{i=1}^n \sum_{1\leq j \neq i \leq n} (K_{ij}-\beta) u_i u_j (z_i^2 - z_i z_j),\\
 & = \sum_{i=1}^n (K_{i0}-\beta) (1-\rho_u) u_i z_i^2 + \sum_{1\leq i\neq j \leq n} (K_{ij}-\beta) u_i u_j \left(\frac{1}{2}z_i^2 + \frac{1}{2}z_j^2- z_i z_j\right),\\
 & \geq 0.\\
\end{align*}
The matrix $\widetilde{M}(u)$ is indeed a semi-definite positive matrix. Hence we have proved inequality \eqref{eq:ineg}, which yields the desired result. 
}

 %%%%%%%%%%%%%%%%%%%%%%%%%%%%%%%%%%%%%%%%%%%%%
\subsection{Proof of Theorem~\ref{th:moving}} \label{sec:proof_exists}

 {For the sake of simplicity, we will prove the existence of a solution $v$ on the finite time interval $[0,T]$ where $T>0$ is an arbitrary positive constant. 
Actually, the proof can be easily adapted to obtain the existence of a global solution for an infinite time horizon.}

\medskip

The proof follows similar lines as the proof of Theorem~2 of~\cite{Jungel} and is divided in three main steps.  {Firstly, an approximate time-discrete problem is introduced for which uniform bounds are proved in a second step. 
Lastly, passing to the limit in this approximate problem using the obtained bounds enables to obtain the existence of a weak solution.}

\medskip

\subsubsection{Step 1 : Approximate time-discrete problem} 

Let us first assume at this point that $\phi_0, \cdots, \phi_n$ belong to $\cC^0([0,T])$. 
\medskip

Let $N\in \N$, $\tau = \frac{T}{N}$ and $\epsilon >0$. For all $k\in \N^*$ so that $k\tau \leq T$, let us denote by $e_k:= e(k \tau)$, $e'_k:= e'(k\tau)$ and 
$\varphi_k = (\phi_{1,k}, \cdots, \phi_{n,k})^T := \varphi(k \tau)$. 
Let us also define
\begin{equation}\label{eq:deffk}
f_k:= \left\{\begin{array}{ll}
\frac{\varphi_k}{e'_k} &  \mbox{ if } e'_k >0,\\
0 & \mbox{ otherwise}, 
             \end{array} \right.
\end{equation}
so that $f_k \in \overline{D}$ and $\varphi_{k} = e'_k f_k$.

By assumption, $w^0(y) := Dh(v^0(y))$ belongs to $L^\infty((0,1); \R^n)$. In the rest of the proof, 
for any $w\in\R^n$, we denote by $v(w):= (Dh)^{-1}(w) = (v_i(w))_{1\leq i \leq n}$ and by $B(w):= M(v(w))$. 

\medskip
 
Let us already mention at this point that the (formal) weak formulation of (\ref{eq:rescaled}) reads as follows: for all $\psi \in L^2((0,T); H^1((0,1); \R^n))$, 
$$
\int_0^T \int_0^1 \partial_t v \cdot \psi + \int_0^T   \int_0^1\partial_y   \frac{1}{e^2} \psi \cdot (A(v) \partial_y v)  + \int_0^T \int_0^1 \frac{e'}{e} (v \cdot \psi + y v \cdot \partial_y \psi)  = \int_0^T\frac{1}{e} \varphi \cdot \psi(\cdot,1).
$$

\medskip

Let us first prove the following lemma.
\begin{lem}\label{lem:intermediate}
Assume that $\phi_0, \cdots, \phi_n \in \cC^0([0,T])$. Then, for all $k \in \N^*$  {such that $k\tau \leq T$}, there exists $w^k \in H^1((0,1); \R^n)$ solution of
\begin{align}\label{eq:approximate2}
&\frac{1}{\tau}\int_0^1 \left( v(w^k) - v(w^{k-1}) \right) \cdot \psi + \frac{1}{e_k^2}\int_0^1 \partial_y \psi \cdot (B(w^k) \partial_y w^k) + \epsilon \int_0^1 (\partial_y w^k \cdot \partial_y \psi + w^k\cdot \psi) \\ \nonumber
& + \frac{e'_k}{e_k} \int_0^1 (v(w^k) \cdot \psi + y v(w^k) \cdot \partial_y \psi) = \frac{1}{e_k}\varphi_k \cdot \psi(1),\\ \nonumber
\end{align}
for all $\psi \in H^1((0,1); \R^n)$. Besides, the following discrete inequality holds for all $k\in\N^*$  {such that $k\tau \leq T$},
\begin{align}\label{eq:entropy}
 &  \frac{1}{\tau} \int_0^1 h(v(w^k)) + \epsilon \int_0^1 \left(|\partial_y w^k|^2 + |w^k|^2\right) + \frac{1}{e_k^2}\int_0^1 \partial_y w^k \cdot (B(w^k) \partial_y w^k)\\ \nonumber
& \leq \frac{1}{\tau} \int_0^1 h(v(w^{k-1})) + \frac{e'_k}{e_k} \left( h(f_k) - \int_0^1 h(v(w^k)) \right). \\ \nonumber
\end{align}
\end{lem}

The proof of this lemma is postponed until Section~\ref{proof:lemma2}. Let us point out the following fact: from (\ref{eq:entropy}), we obtain 
\begin{align}\label{eq:entropyfin}
 &  \left( \frac{1}{\tau} + \frac{e'_k}{e_k} \right) \int_0^1 h(v(w^k)) + \epsilon \int_0^1 (|\partial_y w^k|^2 + |w^k|^2) + \frac{1}{e_k^2}\int_0^1 \partial_y w^k \cdot B(w^k) \partial_y w^k\\ \nonumber
& \leq \frac{1}{\tau} \int_0^1 h(v(w^{k-1})) + \frac{e'_k}{e_k} \|h\|_{L^\infty(\overline{\cD})},  \\ \nonumber
\end{align}
which implies
\begin{align}\label{eq:entropyfin2}
 & \frac{1}{\tau} \int_0^1 h(v(w^k)) + \epsilon \int_0^1 (|\partial_y w^k|^2 + |w^k|^2) + \frac{1}{e_k^2}\int_0^1 \partial_y w^k \cdot B(w^k) \partial_y w^k\\ \nonumber
& \leq \frac{1}{\tau} \int_0^1 h(v(w^{k-1})) + 2 \frac{e'_k}{e_k} \|h\|_{L^\infty(\overline{\cD})}.  \\ \nonumber
\end{align}

\subsubsection{Step 2: Uniform bounds}

For all $0\leq i \leq n$, let $(\phi_{i,p})_{p\in \N}$ be a sequence of  {non-negative} functions of $\cC^0([0,T])$ which weakly-* converges to $\phi_i$ in $L^{\infty}(0,T)$ as $p$ goes to infinity, 
and for all $p\in \N$, 
$$
\|\phi_{i,p}\|_{L^\infty(0,T)} \leq \|\phi_i\|_{L^\infty(0,T)}.
$$
Let us define
$$
\varphi_p:= (\phi_{1,p}, \cdots, \phi_{n,p})^T, \quad \mbox{ and } e_p(t):= e_0 + \int_0^t \sum_{i=0}^n \phi_{i,p}(s)\,ds.
$$
 {It holds that $(e_p)_{p\in\N^*}$ strongly converges to $e$ in $L^\infty(0,T)$. Indeed, let $\varepsilon >0$. Since $e$ is continuous on $[0,T]$, it is uniformly continuous, and there exists $\eta>0$ so that for all 
$t,t'\in [0,T]$ satisfying $|t-t'|\leq \eta$, then $|e(t) - e(t')| \leq \varepsilon/2$. Let $M\in\N^*$ and $0 = s_0 < s_1 < \cdots < s_M = T$ so that for all $0\leq j \leq M-1$, $|s_j - s_{j+1}|  \leq \eta$. Then, it holds that 
$$
\mathop{\max}_{0\leq j \leq M} |e_p(s_j) - e(s_j)| \mathop{\longrightarrow}_{p\to +\infty } 0,
$$
because of the weak-* convergence in $L^\infty[0,T]$ of $(\phi_{i,p})_{p\in\N^*}$ to $\phi_i$ for all $0\leq i \leq n$. }

 {Thus, there exists $p_0\in \N^*$ large enough such that for all $p\geq p_0$, $\dps \mathop{\max}_{0\leq j \leq M} |e_p(s_j) - e(s_j)| \leq \varepsilon/2$. Besides, the non-negativity of the functions 
$\phi_i$ and $\phi_{i,p}$ implies that $e$ and $e_p$ are non-decreasing functions, so that for all $0\leq j \leq M-1$ and all $p\in\N^*$, 
$$
\forall s\in [s_j, s_{j+1}], \quad \quad e(s_j) \leq e(s) \leq e(s_{j+1}) \quad \mbox{ and } \quad e_p(s_j) \leq e_p(s) \leq e_p(s_{j+1}). 
$$
As a consequence, for all $p\geq p_0$, all $0\leq j \leq M-1$ and all $s\in [s_j,s_{j+1}]$,  
\begin{align*}
|e(s) - e_p(s)| & \leq \max\left( |e(s_{j+1})-e_p(s_j)|, |e_p(s_{j+1}) - e(s_j)| \right)\\
& \leq \max\left( |e(s_{j+1}) - e(s_j)| + |e(s_j) - e_p(s_j)|, |e_p(s_{j+1}) - e(s_{j+1})| + |e(s_{j+1}) - e(s_j)|\right) \leq \varepsilon.\\
\end{align*}
Hence, for all $p\geq p_0$, $\|e - e_p\|_{L^\infty(0,T)}\leq \varepsilon$, which yields the strong convergence of the sequence $(e_p)_{p\in\N^*}$ to $e$ in $L^\infty(0,T)$.}

\medskip

For all $k\in\N^*$  {such that $k\tau \leq T$}, we denote by $w^{k, p}$ a solution to (\ref{eq:approximate2}) associated to the fluxes $(\phi_{i,p})_{0\leq i \leq n}$. 
The time-discretized associated quantities are denoted (using obvious notation) by $\varphi_{k,p}$, $e_{k,p}$ and $e'_{k,p}$.

\medskip

Let us define the piecewise constant in time functions $w^{(\epsilon, \tau,p)}(y,t)$, 
$v^{(\epsilon, \tau,p)}(y,t)$, $\sigma_\tau v^{(\epsilon, \tau,p)}(y,t)$, $e_{(\tau,p)}(t)$ and $e^d_{(\tau,p)}(t)$ 
as follows: for all $k\geq 1$  {such that $k\tau \leq T$}, $(k-1)\tau < t \leq k \tau$ and almost all $y\in(0,1)$,  
\begin{align*}
& w^{(\epsilon, \tau,p)}(y,t) := w^{k,p}(y),  \quad v^{(\epsilon, \tau,p)}(y,t) := Dh(w^{k,p}(y)),  \quad \sigma_\tau v^{(\epsilon, \tau,p)}(y,t) = Dh(w^{k-1,p}(y)),\\
& e_{(\tau,p)}(t)  = e_{k,p}, \quad e^d_{(\tau,p)}(t) := e'_{k,p},  \quad \varphi_{(\tau,p)} := \varphi_{k,p}.\\
\end{align*}
Besides, let us set $w^{(\epsilon, \tau,p)}(0, \cdot) = Dh(v^0)$ and $v^{(\epsilon, \tau,p)}(0, \cdot) = v^0$. Let us also denote by $(v_1^{(\epsilon, \tau,p)}, \cdots, v_n^{(\epsilon, \tau,p)})$ 
the $n$ components of $v^{(\epsilon, \tau,p)}$.

\medskip

Then, the following system holds for all piecewise constant in time functions $\psi: (0,T) \to H^1((0,1);\R^n)$, 
\begin{align}\label{eq:approximatetau}
& \frac{1}{\tau}\int_0^T \int_0^1 \left( v^{(\epsilon, \tau,p)} - \sigma_\tau v^{(\epsilon, \tau,p)}) \right) \cdot \psi \,dy\,dt + \int_0^T \frac{1}{e_{(\tau,p)}^2}\int_0^1 \partial_y \psi \cdot (B(w^{(\epsilon, \tau,p)})
\partial_y w^{(\epsilon, \tau,p)}) \, dy\,dt \\ \nonumber
& + \epsilon \int_0^T \int_0^1 (\partial_y w^{(\epsilon, \tau,p)} \cdot \partial_y \psi + w^{(\epsilon, \tau,p)}\cdot \psi)\,dy \,dt  + \int_0^T \frac{e^d_{(\tau,p)}}{e_{(\tau,p)}} \int_0^1 v(w^{(\epsilon, \tau,p)}) \cdot \psi + 
y v(w^{(\epsilon, \tau,p)}) \cdot \partial_y \psi)\,dy\,dt \\ \nonumber
& = \int_0^T \frac{1}{e_{(\tau,p)}}\varphi_{(\tau,p)} \cdot \psi(1) \,dt.\\ \nonumber
\end{align}
The set of piecewise constant functions in time $\psi : (0,T) \to H^1((0,1);\R^n)$ is dense in $L^2((0,T); H^1((0,1); \R^n))$, so that (\ref{eq:approximatetau}) also holds for any 
$\psi \in L^2((0,T); H^1((0,1);\R^n))$. 

\medskip

Using the fact that $A$ satisfies assumption~(H2) of Theorem~\ref{th:Jungel} and the fact that $\partial_y w^{k,p} = D^2h(v^{k,p})\partial_y v^{k,p}$, we obtain for all $k\in\N^*$  {such that $k\tau \leq T$},
\begin{align*}
\int_0^1 \partial_y w^{k,p} \cdot (B(w^{k,p}) \partial_y w^{k,p}) & = \int_0^1 \partial_y v(w^{k,p})\cdot \left[ D^2h(v(w^{k,p})) A(v(w^{k,p})) \partial_y v(w^{k,p}) \right]\,dy\\
& \geq \sum_{i=1}^n \int_0^1  {\alpha \left| v_i(w^{k,p})\right|^{2m_i-2}}|\partial_y v_i(w^{k,p})|^2\,dy = \sum_{i=1}^n \int_0^1 { |\partial_y G_i(v_i(w^{k,p}))|^2}\, dy\\
& = \int_0^1  {|\partial_y G(v(w^{k,p}))|^2}\, dy,\\
\end{align*}
where  {$G_i(s):= \frac{\sqrt{\alpha}}{m_i} |s|^{m_i}$ for all $s\in (0,1)$ and $G(z) =\left(G_i(z_i)\right)_{1\leq i \leq n}$ for all $z:=(z_i)_{1\leq i \leq n} \in (0,1)^n$.} It follows from (\ref{eq:entropyfin2}) that 
for all $k\in\N^*$  { such that $k\tau \leq T$}, 
\begin{align*}
 & \int_0^1 h(v(w^{k,p})) + \tau \int_0^1 |\partial_y \widetilde{\alpha}(v(w^{k,p}))|^2 \\
 & + \epsilon \tau \int_0^1 \left(|\partial_y w^{k,p}|^2 +  |w^{k,p}|^2\right) \leq 2 \tau \|h\|_{L^\infty(\overline{\cD})} \frac{e'_{k,p}}{e_{k,p}} + \int_0^1 h(v(w^{k-1,p})).\\ 
\end{align*}

Summing these inequalities yields, for $k\in\N^*$ so that $k\tau \leq T$,  
\begin{align}\label{eq:ineq00}
 & \int_0^1 h(v(w^{k,p}))+ \tau \sum_{j=1}^k \int_0^1  {|\partial_y G(v(w^{j,p}))|^2}  + \epsilon \tau \sum_{j=1}^k \int_0^1 (|\partial_y w^{j,p}|^2 +  |w^{j,p}|^2) \\ \nonumber
 &\leq 2 \tau \|h\|_{L^\infty(\overline{\cD})} \sum_{j=1}^k \frac{e'_{j,p}}{e_{j,p}} + \int_0^1 h(v^0),\\ \nonumber
 & \leq 2 \|h\|_{L^\infty(\overline{\cD})}\frac{1}{e_0} \sum_{j=1}^k \tau e'_{j,p} + \int_0^1 h(v^0),\\ \nonumber
 & \leq  2 \|h\|_{L^\infty(\overline{\cD})}\frac{(n+1)\|\Phi\|_{L^\infty(0,T)}}{e_0} T + \int_0^1 h(v^0).\\ \nonumber
\end{align}

\medskip

In the sequel, $C$ will denote an arbitrary constant, which may change along the calculations, 
but remains independent on $\epsilon$, $\tau$, $p$ and $\Phi$. We are deliberately keeping here the explicit dependence of the constants on $\|\Phi\|_{L^\infty(0,T)}$ in view of the proof of 
Proposition~\ref{prop:optimal}. It then holds that 
$$
\|e^d_{(\tau, p)}\|_{L^\infty(0,T)} \leq C\|\Phi\|_{L^\infty(0,T)} \mbox{ and } 0< e_0 \leq \|e_{(\tau, p)}\|_{L^\infty(0,T)} \leq C \|\Phi\|_{L^\infty(0,T)}.
$$
We also obtain from (\ref{eq:ineq00})  {and the fact that $\|G_i\|_{L^\infty(0,1)} \leq \frac{\sqrt{\alpha}}{m_i}$ for all $1\leq i \leq n$} that 
\begin{equation}\label{eq:estimalpha}
 {\| G(v^{(\epsilon, \tau, p)}) \|_{L^2((0,T); H^1(0,1)^n)}} \leq C\left( 1 + \|\Phi\|_{L^\infty(0,T)}\right)
\end{equation}
and
\begin{equation}\label{eq:estimw}
\sqrt{\epsilon} \|w^{(\epsilon, \tau, p)}\|_{L^2((0,T); H^1(0,1)^n)} \leq C\left( 1 +  \|\Phi\|_{L^\infty(0,T)}\right).
\end{equation}
Since for all $1\leq i \leq n$, $m_i\leq 1$, this implies that
\begin{align}\label{eq:H1bound}
\|\partial_y v_i^{(\epsilon, \tau, p)} \|_{L^2((0,T); L^2(0,1))} & =   {\left\|\frac{\left|v_i^{(\epsilon, \tau, p)}\right|^{1-m_i}}{m_i} \partial_y \left(|v_i^{(\epsilon, \tau, p)}|^{m_i}\right)\right\|_{L^2((0,T); L^2(0,1))}}\\ \nonumber 
& =  {\left \|\frac{\left|v_i^{(\epsilon, \tau, p)}\right|^{1-m_i}}{\sqrt{\alpha}} \partial_y G_i(v_i^{(\epsilon, \tau, p)})\right\|_{L^2((0,T); L^2(0,1))}}\\ \nonumber 
& \leq C \| \partial_y  {G_i}(v_i^{(\epsilon, \tau, p)})\|_{L^2((0,T); L^2(0,1))} \leq C\left( 1 + \|\Phi\|_{L^\infty(0,T)}\right).  \\  \nonumber 
\end{align}
Besides, 
\begin{align}\label{eq:Abound}
 \left\| A(v^{(\epsilon, \tau, p)} \partial_y v^{(\epsilon, \tau, p)}\right\|_{L^2((0,T); L^2(0,1)^n)}^2 & \leq   \left\| A(v^{(\epsilon, \tau, p)} )\right\|^2_{L^\infty((0,T); L^\infty(0,1)^{n\times n})} \left\|\partial_y v^{(\epsilon, \tau, p)}\right\|_{L^2((0,T); L^2(0,1)^n)}^2 \\ \nonumber
 & \leq C\left( 1 + \|\Phi\|_{L^\infty(0,T)}\right), \\ \nonumber
\end{align}
using the fact that $A\in \cC^0(\overline{\cD}; \R^{n\times n})$. 

\medskip

This yields that for all $\psi \in L^2((0,T); H^1((0,1); \R^n))$, 
\begin{align*}
 \frac{1}{\tau}\left| \int_{\tau}^T \int_0^1 (v^{(\epsilon, \tau, p)} - \sigma_{\tau}v^{(\epsilon, \tau, p)})\cdot \psi \,dy\,dt\right| & 
 \leq \frac{1}{e_0^2}\| A(v^{(\epsilon, \tau, p)} \partial_y v^{(\epsilon, \tau, p)}\|_{L^2((0,T); L^2(0,1)^n)}\|\partial_y \psi\|_{L^2((0,T); L^2(0,1)^n)}\\
 & + \epsilon \|w^{(\epsilon, \tau, p)}\|_{L^2((0,T); H^1(0,1)^n)} \|\psi\|_{L^2((0,T); H^1(0,1)^n)} \\
 & + 2 \frac{\|e^d_{(\tau, p)}\|_{L^\infty(0,T)}}{e_0} \|v^{(\epsilon, \tau, p)}\|_{L^2((0,T); H^1(0,1)^n)}\|\psi\|_{L^2((0,T); H^1(0,1)^n)}\\
 & + \frac{1}{e_0} \|\Phi\|_{L^\infty(0,T)} \|\psi\|_{L^2((0,T); H^1(0,1)^n)},\\
 & \leq C\left( 1  + \left\| \Phi\right\|_{L^\infty(0,T)}\right)\|\psi\|_{L^2((0,T); H^1(0,1)^n)}.\\
\end{align*}
This last inequality shows that 
\begin{equation}\label{eq:timeder}
\frac{1}{\tau}\|v^{(\epsilon,\tau,p)} - \sigma_{\tau} v^{(\epsilon, \tau,p)} \|_{L^2((\tau,T); \left(H^1(0,1)^n\right)')} \leq C\left( 1  + \left\| \Phi\right\|_{L^\infty(0,T)}\right).
\end{equation}

\subsubsection{Step 3: The limit $p\to +\infty$ and $\epsilon, \tau \to 0$}

For all $p\in\N^*$, the functions $e'_p$ and $e_p$ are continuous on $[0,T]$, and hence are uniformly continuous. As a consequence, there exists $\tau_p >0$ small enough so that for any $t,t'\in [0,T]$ 
satisfying $|t-t'| \leq \tau_p$, then $|e'_p(t) - e'_p(t')| \leq \frac{1}{p}$ and $|e_p(t) - e_p(t')| \leq \frac{1}{p}$. This implies in particular that 
$$
\|e^d_{(\tau_p, p)} - e'_p\|_{L^\infty(0,T)} \leq \frac{1}{p} \quad \mbox{ and } \|e_{(\tau_p, p)} - e_p\|_{L^\infty(0,T)} \leq \frac{1}{p}.
$$
 {These inequalities, together with the fact that $(e'_p)_{p\in\N^*}$ weakly-* converges to $e'$ in $L^\infty(0,T)$ (respectively that $(e_p)_{p\in\N^*}$ strongly converges to $e$ in $L^\infty(0,T)$), 
imply that the sequence $\left(e^d_{(\tau_p,p)}\right)_{p\in\N^*}$ (respectively $\left(e_{(\tau_p, p)}\right)_{p\in\N^*}$) also weakly-* converges to $e'$ in $L^\infty(0,T)$ (respectively strongly converges to $e$ in $L^\infty(0,T)$).}

\medskip

In the following, we consider such a subsequence $(\tau_p)_{p\in\N^*}$. The uniform estimates (\ref{eq:timeder}) and (\ref{eq:H1bound}) 
allow us to apply the Aubin lemma in the version of Theorem~1 of~\cite{AubinJungel}. Up to extracting a subsequence which is not relabeled, 
there exists $v = (v_i)_{1\leq i \leq n} \in H^1((0,T); (H^1((0,1); \R^n))') \cap L^2((0,T); H^1((0,1); \R^n))$ so that as $p$ goes to infinity and $\epsilon$ goes to $0$,  
\begin{align*}
v^{(\epsilon, \tau_p, p)} \mathop{\longrightarrow}_{p\to +\infty, \epsilon \to 0} v, & \quad \left\{ \begin{array}{l}
                                                                                             \mbox{ strongly in }L^2((0,T); L^2((0,1);\R^n)),\\
                                                                                             \mbox{ weakly in }L^2((0,T);H^1((0,1);\R^n)),\\
                                                                                             \mbox{ and a.e. in } (0,T)\times (0,1),\\
                                                                                            \end{array} \right.\\
\frac{1}{\tau_p}\left( v^{(\epsilon, \tau_p, p)} - \sigma_{\tau_p} v^{(\epsilon, \tau_p, p)}\right) \mathop{\rightharpoonup}_{p\to +\infty, \epsilon \to 0} \partial_t v & \mbox{ weakly in }L^2((0,T); (H^1((0,1);\R^n))').\\
\end{align*}
Because of the boundedness of $v^{(\epsilon, \tau_p, p)}$ in $L^\infty((0,T); L^\infty((0,1);\R^n))$, the convergence even holds strongly in $L^q((0,T); L^q((0,1);\R^n))$ for any $q<+\infty$, which is a 
consequence of the dominated convergence theorem. The latter theorem, together with $A\in \cC^0(\overline{\cD}; \R^{n\times n})$ 
implies also that the convergence $A(v^{(\epsilon, \tau_p, p)}) \mathop{\longrightarrow} A(v)$ holds strongly in $L^q((0,T); L^q((0,1);\R^{n\times n}))$. 
Moreover, using (\ref{eq:Abound}) and (\ref{eq:estimw}), up to extracting another subsequence, there exists $V\in L^2((0,T); L^2((0,1);\R^n))$ so that 
\begin{align*}
 A(v^{(\epsilon, \tau_p, p)}) \partial_y v^{(\epsilon, \tau_p, p)} \rightharpoonup V & \mbox{ weakly in }L^2((0,T); L^2((0,1);\R^n)),\\
\epsilon w^{(\epsilon, \tau_p, p)} \longrightarrow 0 & \mbox{ strongly in }L^2((0,T); H^1((0,1);\R^n)).\\
\end{align*}
The strong convergence of $A(v^{(\epsilon, \tau_p,p)})$ in $L^q((0,T); L^q((0,1);\R^n))$ and the weak convergence of $\partial_y v^{(\epsilon, \tau_p,p)}$ in $L^2((0,T); L^2((0,1);\R^n))$ 
implies necessarily that $V = A(v)\partial_y v$.

\medskip

We are now in position to pass to the limit $\epsilon \to 0$ and $p\to +\infty$ in (\ref{eq:approximatetau}) with $\tau = \tau_p$ and $\psi \in L^2((0,T); H^1((0,1);\R^n))$. Let us recall that 
$\left(e_{(\tau_p,p)}\right)_{p\in \N^*}$ (respectively $\left( e^d_{(\tau_p, p)}\right)_{p\in\N^*}$) converges strongly (respectively weakly-*) to $e$ (respectively $e'$) in $L^\infty(0,T)$. We obtain that $v$ is a solution to
\begin{align}\label{eq:variat}
& \int_0^T \int_0^1 \partial_t v \cdot \psi \,dy\,dt + \int_0^T \frac{1}{e(t)^2}\int_0^1 \partial_y \psi \cdot ( A(v) \partial_y v) \, dy\,dt \\ \nonumber
& + \int_0^T \frac{e'(t)}{e(t)} \int_0^1 (v \cdot \psi + y v \cdot \partial_y \psi)\,dy\,dt  = \int_0^T \frac{1}{e(t)}\varphi \cdot \psi(1) \,dt,\\ \nonumber
\end{align}
yielding the result.

\subsubsection{Proof of Lemma~\ref{lem:intermediate}}\label{proof:lemma2}

\begin{proof}[Proof of Lemma~\ref{lem:intermediate}]
We prove Lemma~\ref{lem:intermediate} by induction using the Leray-Schauder fixed-point theorem. Let $z \in L^\infty ((0,1); \R^n)$ and $\delta \in [0,1]$. 
We consider the following linear problem: find $w\in H^1((0,1); \R^n)$ solution of 
\begin{equation}\label{eq:linear}
\forall \psi \in H^1((0,1); \R^n), \quad a_z(w, \psi) = l_{\delta, z}(\psi), 
\end{equation}
where 
$$
a_z(w, \psi):= \frac{1}{e_k^2}\int_0^1 \partial_y \psi \cdot B(z) \partial_y w + \epsilon \int_0^1 (\partial_y w \cdot \partial_y \psi + w \cdot \psi)
$$
and 
$$
l_{\delta, z}(\psi):= - \frac{\delta}{\tau}\int_0^1 (v(z) - v(w^{k-1}))\cdot \psi + \frac{\delta}{e_k} \varphi_k \cdot \psi(1) - \delta \frac{e'_k}{e_k}\int_0^1 (v(z) \cdot \psi + y v(z) \cdot \partial_y \psi).
$$
As a consequence of (H2), the matrix $B(z)$ is positive semi-definite for any $z \in \R^n$. Thus, the bilinear form $a_z$ is coercive and continuous on $H^1((0,1); \R^n)$, and it holds that
\begin{equation}\label{eq:coercivity}
\forall \psi \in H^1((0,1); \R^n), \; a_z(\psi, \psi) \geq \epsilon  { \|\psi\|^2_{H^1(0,1)}}.
\end{equation}

Since $v(z) \in L^\infty((0,1); \R^n)$ and $\|v(z)\|_{L^\infty(0,1)} \leq 1$, the linear form $l_{\delta,z}$ is continuous. From the Agmon inequality, there exists $C>0$ independent of 
$\Phi:= (\phi_0, \cdots, \phi_n)$, $\epsilon$ or $\tau$ such that
for all $\psi \in H^1((0,1); \R^n)$, 
\begin{equation}\label{eq:continuity1}
|l_{\delta, z}(\psi) | \leq \left( \frac{2}{\tau} + C \left\| \Phi\right\|_{L^\infty(0,T)}\right) \|\psi\|_{H^1(0,1)}, 
\end{equation}
where $\dps \|\Phi\|_{L^\infty(0,T)} = \mathop{\max}_{i=0, \cdots, n} \|\phi_i\|_{L^\infty(0,T)}$. It immediately follows 
from the Lax-Milgram theorem that there exists a unique solution $w \in H^1((0,1); \R^n)$ to 
(\ref{eq:linear}). 

\medskip

We define the operator $S: [0,1] \times L^\infty((0,1); \R^n)  \to L^\infty((0,1); \R^n)$ as follows. For all $\delta \in [0,1]$ and $\chi \in L^\infty((0,1); \R^n)$, 
$S(\delta, \chi)$ is the unique solution $w\in H^1((0,1); \R^n) \hookrightarrow L^\infty((0,1); \R^n)$ 
of (\ref{eq:linear}). We are going to prove that there exists a fixed-point $w^k \in H^1((0,1); \R^n)$ of the equation $S(1,w^k) = w^k$ using the Leray-Schauder fixed-point theorem (Theorem~\ref{th:LS} in the Appendix). 
This will end the proof of Lemma~\ref{lem:intermediate} since such a fixed-point $w^k$ is a solution of (\ref{eq:approximate2}). 

Let us check that all the assumptions of Theorem~\ref{th:LS} are satisfied: 
\begin{itemize}
 \item[(A1)] For all $\chi \in L^\infty((0,1); \R^n)$, $S(0,\chi) = 0$; 
 \item[(A2)] Let us prove that $S$ is a compact map. To this aim, let us first prove that it is continuous.
 Let $(\delta_n)_{n\in\N}$ and $(\chi_n)_{n\in\N}$ be sequences in $[0,1]$ and $L^\infty((0,1); \R^n)$ respectively, $\delta \in [0,1]$ and $\chi\in L^\infty((0,1); \R^n)$ 
 so that $\dps \delta_n \mathop{\longrightarrow}_{n\to +\infty} \delta$ and $\dps \chi_n \mathop{\longrightarrow}_{n\to +\infty} \chi$ strongly in $L^\infty((0,1); \R^n)$. For all $n\in \N$, 
 let  $w_n:= S(\delta_n, \chi_n)$. From assumption (H1) and the global inversion theorem, $h:\cD \to \R^n$ is a $\cC^2$-diffeomorphism. Thus, together with the fact that $A\in \cC^0(\overline{\cD}; \R^{n\times n})$, 
 it holds that the applications $z\in \R^n \mapsto v(z) = (Dh)^{-1}(z)$ and $z\in \R^n \mapsto B(z) = A(v(z)) D^2h((Dh)^{-1}(z)) = A(v(z) D\left(Dh^{-1}\right)(z)$ are continuous. Hence, $\dps v(\chi_n) \mathop{\longrightarrow}_{n\to +\infty} v(\chi)$ and $\dps B(\chi_n) \mathop{\longrightarrow}_{n\to +\infty} B(\chi)$ strongly 
 in $L^\infty((0,1); \R^n)$ and $L^\infty((0,1); \R^{n\times n})$ respectively. 
 
 Besides, the uniform coercivity and continuity estimates (\ref{eq:coercivity}) and (\ref{eq:continuity1}) imply that 
 $(w_n)_{n\in\N}$ is a bounded sequence in $H^1((0,1); \R^n)$. Thus, up to the extraction of a subsequence which is not relabeled, $(w_n)_{n\in\N}$ weakly converges to some $w$ in $H^1((0,1); \R^n)$. Passing to the 
 limit $n\to +\infty$ in (\ref{eq:linear}) implies that $w = S(\delta, \chi)$. The uniqueness of the limit yields that the whole sequence $(w_n)_{n\in\N}$ weakly converges to $S(\delta, \chi)$ in $H^1((0,1); \R^n)$. 
 The convergence thus holds strongly in $L^\infty((0,1); \R^n)$ because of the compact embedding  {$H^1((0,1); \R^n) \hookrightarrow L^\infty((0,1); \R^n)$}.
 This proves the continuity of the map $S$ and its compactness follows again from the compact embedding  {$ H^1((0,1); \R^n)\hookrightarrow L^\infty((0,1); \R^n)$}.  
 \item[(A3)] Let $\delta \in [0,1]$ and $w\in L^\infty((0,1); \R^n)$ so that $S(\delta , w) = w$. It holds that (taking $\psi = w$ as a test function in (\ref{eq:linear}) with $\chi = w$), 
 \begin{align}
& \frac{1}{e_k^2}\int_0^1 \partial_y w \cdot (B(w) \partial_y w ) + \epsilon \int_0^1 (|\partial_y w|^2 + |w|^2 ) = \\ \label{eq:egalite}
& - \frac{\delta}{\tau}\int_0^1 (v(w) - v(w^{k-1}))\cdot w + \frac{\delta}{e_k} \varphi_k \cdot w(1) - \delta \frac{e'_k}{e_k}\int_0^1 (v(w) \cdot w + y v(w) \cdot \partial_y w).\\ \nonumber
\end{align}
Let us consider separately the different terms appearing in (\ref{eq:egalite}). First, by convexity of $h$, and using the fact that $w = Dh(v(w))$, it holds that
\begin{equation}\label{eq:ineq1}
\frac{\delta}{\tau}\int_0^1 (v(w) - v(w^{k-1}))\cdot w = \frac{\delta}{\tau}\int_0^1 (v(w) - v(w^{k-1}))\cdot Dh(v(w)) \geq \frac{\delta}{\tau} \int_0^1 ( h(v(w)) - h(v(w^{k-1})) ).
\end{equation}
Besides, using an integration by parts, 
\begin{align}\nonumber
& \delta \frac{e'_k}{e_k}\int_0^1 (v(w) \cdot w + y v(w) \cdot \partial_y w) = \delta \frac{e'_k}{e_k} \left( v(w)(1) \cdot w(1) - \int_0^1 y w \cdot \partial_y v(w) \right),\\ \nonumber
&= \delta \frac{e'_k}{e_k} \left( v(w)(1)\cdot Dh(v(w)(1)) - \int_0^1 y Dh(v(w)) \cdot \partial_y v(w) \right),\\ \nonumber
&= \delta \frac{e'_k}{e_k} \left( v(w)(1)\cdot Dh(v(w)(1))- \int_0^1 y \partial_y (h(v(w)))\right),\\ \label{eq:eq1}
&= \delta \frac{e'_k}{e_k} \left( v(w)(1)\cdot Dh(v(w)(1)) - h(v(w)(1)) + \int_0^1 h(v(w)) \right). \\  \nonumber
\end{align}
Using (\ref{eq:deffk}), we obtain 
\begin{equation}\label{eq:eq2}
\frac{\delta}{e_k} \varphi_k \cdot w(1)  = \delta \frac{e_k'}{e_k} f_k \cdot Dh(v(w)(1)).
\end{equation}
Finally, using (\ref{eq:egalite}), (\ref{eq:ineq1}), (\ref{eq:eq1}) and (\ref{eq:eq2}), and again the convexity of $h$, we obtain
\begin{align}\label{eq:ineqentrop}
&  \frac{\delta}{\tau} \int_0^1 h(v(w)) + \epsilon \int_0^1 (|\partial_y w|^2 + |w|^2) + \frac{1}{e_k^2}\int_0^1 \partial_y w \cdot (B(w) \partial_y w)\\ \nonumber
& \leq \frac{\delta}{\tau} \int_0^1 h(v(w^{k-1})) + \delta \frac{e'_k}{e_k} \left( (f_k - v(w)(1)) \cdot Dh(v(w)(1)) + h(v(w)(1)) - \int_0^1 h(v(w)) \right)  \\ \nonumber
& = \frac{\delta}{\tau} \int_0^1 h(v(w^{k-1})) +  \frac{e'_k}{e_k} \left( h(f_k) - \int_0^1 h(v(w)) \right). \\ \nonumber
\end{align}
This inequality implies that
$$
\epsilon \|w\|_{H^1((0,1); \R^n)}^2 \leq \left( \frac{2}{\tau} +C\|\Phi\|_{L^\infty(0,T)}\right) \|h\|_{L^\infty(\overline{D})},
$$
\end{itemize}
for some constant $C>0$ independent of $\epsilon$, $\tau$ of $\Phi$.

All the assumptions of the Leray-Schauder fixed-point theorem are thus satisfied. This yields the existence of a fixed-point solution $w^k\in H^1((0,1); \R^n)$ to $S(1,w^k) = w^k$. Besides, using 
(\ref{eq:ineqentrop}) with 
$\delta = 1$, we have the discrete entropy inequality (\ref{eq:entropy}). 
\end{proof}

\subsection{Proof of Proposition~\ref{prop:longtime}}\label{sec:proof_time}

Let us define by $V:= \sum_{i=0}^n \overline{\phi}_i \in \R_+^*$, $\overline{\varphi}:=( \overline{\phi}_1, \cdots , \overline{\phi}_n)^T$ and $\overline{f}:= \frac{\overline{\varphi}}{V}$. 
From (T1), the vector $\overline{f} := \left( \overline{f}_i\right)_{1\leq i \leq n}$ obviously belongs to the set $\cD$. 

\medskip

If $h$ defined by (\ref{eq:defh}) is an entropy density for which $A$ satisfies assumptions (H1)-(H2)-(H3), then $A$ satisfies the same assumptions with the entropy density 
$$
\overline{h}: \left\{
\begin{array}{ccc}
\cD & \to & \R \\
u & \mapsto & \sum_{i=1}^n u_i \log \frac{u_i}{\overline{f}_i} + (1-\rho_u) \log \frac{1-\rho_u}{1 - \rho_{\overline{f}}}.\\
\end{array}
\right .
$$
Indeed, for all $u\in \cD$, $D\overline{h}(u) = Dh(u) + \overline{g}$, 
where $\overline{g} := \left( \log\left(\frac{1- \rho_{\overline{f}}}{\overline{f}_i}\right)\right)_{1\leq i \leq n}$ is a constant vector in $\R^n$ and $D^2\overline{h}(u) = D^2h(u)$. 
Moreover, the entropy density $\overline{h}$ has the following interesting property: $\overline{f}$ is the unique minimizer of $\overline{h}$ on $\overline{\cD}$ so that $\overline{h}(u) \geq \overline{h}(\overline{f}) = 0$ for all 
$u\in \overline{\cD}$. In the rest of the proof, for all $w\in \R^n$, we will denote by $\overline{v}(w) = \left( \overline{v}_i(w)\right)_{1\leq i \leq n} := (D\overline{h})^{-1}(w) =Dh^{-1}(w - \overline{g})$. 

\medskip

Let $\left(\overline{w}^{\epsilon,k}\right)_{k\in\N}$ be a sequence of solutions to the regularized time-discrete problems (\ref{eq:approximate2}) defined in Lemma~\ref{lem:intermediate} 
with the constant fluxes $(\overline{\phi}_0, \cdots, \overline{\phi}_n)$ and the entropy density $\overline{h}$. The entropy inequality (\ref{eq:entropy}) then reads
\begin{align}\label{eq:entropy2}
& \frac{1}{\tau} \int_0^1 \overline{h}(\overline{v}(\overline{w}^{\epsilon,k})) + \epsilon \int_0^1 (|\partial_y \overline{w}^{\epsilon,k}|^2 + |\overline{w}^{\epsilon,k}|^2) + \frac{1}{e_k^2}\int_0^1 \partial_y \overline{w}^{\epsilon,k} \cdot B(\overline{w}^{\epsilon,k}) \partial_y \overline{w}^{\epsilon,k}\\ \nonumber
& \leq \frac{1}{\tau} \int_0^1 \overline{h}(\overline{v}(\overline{w}^{\epsilon,k-1})) + \frac{e'_k}{e_k} \left(\overline{h}(\overline{f}) - \int_0^1 \overline{h}(\overline{v}(\overline{w}^{\epsilon,k})) \right). \\ \nonumber
\end{align}
In our particular case, for all $k\in\N$, $e'_k= V$, $e_k = e_0 + Vk\tau$ and $\overline{h}(\overline{f}) = 0$, so that we obtain
$$
\frac{e_0 + V (k+1)\tau }{\tau} \int_0^1 \overline{h}(\overline{v}(\overline{w}^{\epsilon,k})) - \frac{e_0 + V k\tau }{\tau} \int_0^1 \overline{h}(\overline{v}(\overline{w}^{\epsilon,k-1})) \leq 0.
$$
This implies that for all $k\in \N$ and $\epsilon >0$, 
\begin{equation}\label{eq:ineqtime}
(e_0 + V (k+1)\tau) \int_0^1  \overline{h}(\overline{v}(\overline{w}^{\epsilon,k})) \leq (e_0 + V \tau)\int_0^1 \overline{h}(\overline{v}(w^{0})).
\end{equation}
Let us denote by $\overline{w}^{(\epsilon,\tau)}: \R_+^* \to H^1((0,1); \R^n)$ the piecewise constant in time function defined by
$$
\mbox{ for a.a. }y\in (0,1), \quad \overline{w}^{(\epsilon,\tau)}(t,y) = \overline{w}^{\epsilon,k}(y) \mbox{ if } (k-1)\tau < t \leq k\tau.
$$
Let $T>0$ and $\xi \in L^1(0,T)$ such that $\xi \geq 0$ a.e. in $(0,T)$. Inequality (\ref{eq:ineqtime}) and Fubini's theorem for integrable functions implies that
$$
\int_0^T\int_0^1 \left[ (e_0 + V (k+1)\tau) \overline{h}(\overline{v}(\overline{w}^{(\epsilon,\tau)}))    - (e_0 + V \tau) \overline{h}(\overline{v}(\overline{w}^{0})) \right] \xi(t)\,dy\,dt \leq 0.
$$
From the proof of Theorem~\ref{th:moving}, we know that up to the extraction of a subsequence which is not relabeled, $\left(\overline{v}(\overline{w}^{(\epsilon,\tau)})\right)_{\epsilon, \tau >0}$ converges strongly in 
$L^2_{\rm loc}(\R_+^*; L^2((0,1); \R^n))$ and a.e. in $\R_+^* \times (0,1)$ as $\epsilon$ and $\tau$ go to zero to a global weak solution $v$ to (\ref{eq:rescaled}).
Using Lebesgue dominated convergence theorem, and passing to the limit $\epsilon,\tau\to 0$ in the above inequality yields
$$
\int_0^T\int_0^1 \left[ (e_0 + V t)\overline{h}(v) - e_0 \overline{h}(\overline{v}(w^{0})) \right] \xi(t)\,dy\,dt \leq 0,
$$
which implies that there exists $C>0$ such that for almost all $t>0$, 
\begin{equation}\label{eq:ineq2}
(e_0 + Vt) \int_0^1 \overline{h}(v)  \leq C, 
\end{equation}
 {which yields inequality (\ref{eq:relentropy}).} In the rest of the proof, $C$ will denote an arbitrary positive constant independent on the time $t>0$. Furthermore, since 
$v\in H^1((0,T); (H^1((0,1); \R^n))') \cap L^2((0,T); H^1((0,1);\R^n))$, it holds that $v\in \cC^0((0,T); L^2((0,1);\R^n))$ from~\cite{LionsMagenes}, and the Lebesgue dominated convergence theorem implies that 
 $t\in \R_+^* \mapsto \int_0^1 \overline{h}(v(t,y))\,dy $ is a continuous function. Inequality (\ref{eq:ineq2}) then holds for all $t>0$.
 
For all $0\leq i \leq n$, let us denote by $\overline{v_i}(t):= \int_0^1 v_i(t,y)\,dy$. By convention, we define $v_0(t,y):= 1 - \rho_{v(t,y)}$ and $\overline{f}_0:= 1 - \rho_{\overline{f}}$. 
It can be checked from the weak formulation of (\ref{eq:approximate2}) that 
$$
\int_0^1 \overline{v}_i\left( \overline{w}^{\epsilon,k}\right) = \frac{k \overline{\phi}_i\tau + e_0 \int_0^1 v_i^0}{e_0 + V(k+1)\tau}.
$$
Passing to the limit $\epsilon,\tau \to 0$ using the Lebesgue dominated convergence theorem, we obtain that for almost all $t>0$,  
$$
\overline{v}_i(t) = \frac{e_0 \int_0^1 v_i^0(y)\,dy + t \overline{\phi}_i}{e_0 + Vt},
$$
so that $ \dps |\overline{v}_i(t) - \overline{f}_i| \leq \frac{C}{e_0 + Vt}$. The continuity of $\overline{v}_i$ implies that this equality holds for all $t>0$.

The Csiz\`ar-Kullback inequality states that for all $t>0$,
 {
$$
\left\| v_i(t,\cdot) - \overline{v}_i(t)\right\|^2_{L^1(0,1)} \leq 2\int_0^1 v_i(t,y) \log \frac{v_i(t,y)}{\overline{v}_i(t)}\,dy =  2 \int_0^1 v_i(t,y) \log \frac{v_i(t,y)}{\overline{f}_i}\,dy + 2 \int_0^1 v_i(t,y) \log \frac{\overline{f}_i}{\overline{v}_i(t)}\,dy.
$$
 Thus,
 \begin{align*}
 \sum_{i=0}^n  \left\| v_i(t,\cdot) - \overline{f}_i\right\|_{L^1(0,1)} & \leq \sum_{i=0}^n \left\| v_i(t,\cdot) -\overline{v}_i(t) \right\|_{L^1(0,1)} + |\overline{f}_i - \overline{v}_i(t)|\\
 & \leq \sqrt{2 \int_0^1\overline{h}(v)} +  \sum_{i=0}^n \left[ \sqrt{2\left|\log \frac{\overline{v}_i(t)}{\overline{f}_i}\right|} + |\overline{f}_i - \overline{v}_i(t)|\right]\\
 & \leq \sqrt{\frac{C}{e_0 + Vt}}.\\
 \end{align*}
 Hence inequality (\ref{eq:L1}) and} the result.
 
 %%%%%%%%%%%%%%%%%%%%%%%%%%%%%%%%%%%%%%%%%%%%%
\subsection{Proof of Proposition~\ref{prop:optimal}}\label{sec:proof_optimal}
Let $(\Phi^m)_{m \in \N} \subset \Xi$ be a minimizing sequence for $\cJ$ i.e such that
$$
\mathop{\lim}_{m\to +\infty} \cJ(\Phi^m) = \mathop{\inf}_{\Phi \in \Xi} \cJ (\Phi). 
$$
By definition of the set $\Xi$, the sequence $(\Phi^m)_{m\in \N}$ is bounded in $L^\infty(0,T)$. Thus, up to a non relabeled extraction, 
it weakly-* converges to some limit $\Phi^*\in \Xi$ in $L^\infty(0,T)$. As a consequence, $\left(\frac{d}{dt} e_{\Phi^m}\right)_{m\in\N}$ (respectively $\left(e_{\Phi^m}\right)_{m\in\N}$) converges weakly-* (respectively strongly)
in $L^\infty(0,T)$ to $\frac{d}{dt}e_{\Phi^*}$ (respectively $e_{\Phi^*}$).

\medskip

For each $m\in \N$, let $v_{\Phi^m}$ be the unique global weak solution to (\ref{eq:rescaled}) associated to the fluxes $\Phi^m$. Its uniqueness is a consequence of assumption (C1). 
From the bounds obtained in the proof of Theorem~\ref{th:moving} and the boundedness 
of $(\Phi^m)_{m\in \N}$ in $L^\infty(0,T)$, it holds that the sequences $\|\partial_t v_{\Phi^m}\|_{L^2((0,T); (H^1(0,1))'}$, $\left\|A(v_{\Phi^m})\partial_y v_{\Phi^m}\right\|_{L^2((0,T); L^2(0,1))}$ 
and $\|\partial_y v_{\Phi^m}\|_{L^2((0,T); L^2(0,1))}$ are also uniformly bounded in $m$.

Thus, up to the extraction of a subsequence which is not relabeled, using the compact injection of $L^2((0,T); H^1((0,1);\R^n)) \cap H^1((0,T); (H^1((0,1);\R^n))')$ into $\cC((0,T); L^2((0,1);\R^n))$ 
(see~\cite{LionsMagenes}), there exists $v_*\in L^2((0,T); H^1((0,1);\R^n)) \cap H^1((0,T); (H^1((0,1);\R^n))')$ and $V_*\in L^2((0,T); L^2((0,1);\R^n))$ so that
\begin{align*}
 v_{\Phi^m} \mathop{\rightharpoonup} v_* & \mbox{ weakly in } L^2((0,T); H^1((0,1);\R^n)) \cap H^1((0,T); (H^1((0,1);\R^n))'), \\
 v_{\Phi^m} \mathop{\longrightarrow} v_* & \mbox{ strongly in }\cC((0,T); L^2((0,1);\R^n)) \mbox{ and a.e. in }(0,T)\times (0,1),\\ 
A(v_{\Phi^m})\partial_y v_{\Phi^m} \mathop{\rightharpoonup} V_* & \mbox{ weakly in }L^2((0,T); L^2((0,1);\R^n)).\\
\end{align*}
Using similar arguments as in the proof of Theorem~\ref{th:moving}, we also obtain that $V_*$ is necessarily equal to 
$A(v_*)\partial_y v_*$. Passing to the limit $m\to +\infty$, we obtain that for all $\psi\in L^2((0,T); H^1((0,1);\R^n))$, 
\begin{align*}
& \int_0^T \int_0^1 \partial_t v_* \cdot \psi\,dt\,dy + \int_0^T \int_0^1 \frac{1}{e_{\Phi^*}(t)^2} \partial_y  \psi \cdot (A(v_*) \partial_y v_*)\,dt\,dy \\
&+ \int_0^T \frac{\frac{d}{dt}e_{\Phi^*}(t)}{e_{\Phi^*}(t)}\int_0^1 (v_* \cdot \psi + y v_* \cdot \partial_y \psi)\,dt\,dy  = \int_0^T\frac{1}{e_{\Phi^*}(t)} \varphi_*(t) \cdot \psi(1)\,dt.\\
\end{align*}
Assumption (C1) yields $v_* = v_{\Phi^*}$. The above convergence results then imply that
$$
\mathcal{J}\left( \Phi^m\right) \mathop{\longrightarrow}_{m\to +\infty} \mathcal{J}(\Phi^*), 
$$
and hence $\Phi^*$ is a minimizer of problem (\ref{eq:optimal}). Hence the result.

%%%%%%%%%%%%%%%%%%%%%%%%%%%%%%%%%%%%%%%%%%%%%
%%%%%%%%%%%%%%%%%%%%%%%%%%%%%%%%%%%%%%%%%%%%%
\section{Numerical tests}\label{sec:numerics}

In this section, we present some numerical tests illustrating the results of Section~\ref{sec:flux} on the prototypical example of Section~\ref{sec:example}. In Section~\ref{sec:scheme}, we present the numerical scheme used in our simulations 
to compute an approximation of a solution of~(\ref{eq:rescaled}). In Section~\ref{sec:longtimeres} and Section~\ref{sec:optimalres}, some numerical tests which illustrate Proposition~\ref{prop:longtime} and 
Proposition~\ref{prop:optimal} are detailed. 

\subsection{Discretization scheme}\label{sec:scheme}

In view of the optimization problem (\ref{eq:optimal}) we are aiming at, it appears that a fully implicit unconditionally stable scheme is needed to allow the use of reasonably large time steps.

\medskip

We present here the numerical scheme used for the discretization of~(\ref{eq:rescaled}), for the particular model presented in Section~\ref{sec:example}. We do not provide a rigorous numerical analysis for this scheme here. 

Let $M\in\N^*$ and $\Delta t:= \frac{T}{M}$. We define for all $0\leq m \leq M$, $t_m:= m \Delta t$. The discrete external fluxes are characterized for every 
$0\leq i \leq n$ by vectors  $\widehat{\phi}_i:= \left(\widehat{\phi}_i^m\right)_{1\leq m \leq M} \in \R^M_+$, where 
$\widehat{\phi}_i^m  = \int_{t_{m-1}}^{t_m} \phi_i(s)\,ds$.
For every  $1 \leq m \leq M$, the thickness of the thin film and it derivative at time $t_m$ are approximated respectively by 
$$
e_m := e_0 + \sum_{p=1}^m  \sum_{i=0}^n \widehat{\phi}_i^p \Delta t  \approx e(t_m), \quad \mbox{ and } e^d_m :=  \sum_{i=0}^n \widehat{\phi}_i^m \approx e'(t_m). 
$$

\medskip

In addition, let $Q\in \N^*$ and $\Delta y:= \frac{1}{Q}$ and $y_q:= (q-0.5)\Delta y$. For all $0\leq i \leq n$, $1\leq q\leq Q$ and $0\leq m \leq M$, 
we denote by $v_i^{m,q}$ the finite difference approximation of $v_i$ at time $t_m$ and point $y_q\in(0,1)$. Here again, we use the convention that $v_0 = 1 - \rho_v$. 

We use a centered second-order finite difference scheme for the diffusive part of the equation, and 
a first-order upwind scheme for the advection part, together with a fully implicit time scheme. Assuming that the approximation 
$\left(v_i^{m-1,q}\right)_{0\leq i \leq n, 1\leq q\leq Q}$ is known, one computes $\left(\widetilde{v}_i^{m,q}\right)_{0\leq i \leq n,~ 1\leq q\leq Q}$ as solutions of the following sets of equations.

\medskip

%\small
For all $0\leq i \leq n$ and $2\leq q\leq Q-1$, 
\begin{align}\label{eq:bulk}
 \frac{\left( \widetilde{v}_i^{m,q} - v_i^{m-1,q}\right) }{\Delta t} &=  \frac{e^d_m}{e_m} y_q  \left(  \frac{\widetilde{v}^{m,q+1}_i -\widetilde{v}^{m,q}_i}{\Delta y} \right) \\ \nonumber
 &+  \sum_{0\leq j \neq i \leq n}  \frac{K_{ij}}{e_m^2} \left[\widetilde{v}^{m,q}_j  \left( \frac{\widetilde{v}^{m,q+1}_i +\widetilde{v}^{m,q-1}_i -2 \widetilde{v}^{m,q}_j}{2\Delta y^2}\right)   - \widetilde{v}^{m,q}_i \left( \frac{ \widetilde{v}^{m,q+1}_j +\widetilde{v}^{m,q-1}_j -2 \widetilde{v}^{m,q}_j}{2\Delta y^2}\right)  \right] \\ \nonumber
\end{align}
\normalsize
together with boundary conditions which reads for all $0\leq i \leq n$,  
\begin{align}\label{eq:BC}
\sum_{0\leq j \neq i \leq n}  \frac{K_{ij}}{e_m} \left[\widetilde{v}^{m,1}_j  \left( \frac{\widetilde{v}^{m,2}_i - \widetilde{v}^{m,1}_i }{\Delta y}\right)   - \widetilde{v}^{m,1}_i \left( \frac{ \widetilde{v}^{m,2}_j -\widetilde{v}^{m,1}_j}{\Delta y}\right)  \right] &  = 0, \\ \label{eq:BC1}
\sum_{0\leq j \neq i \leq n}  \frac{K_{ij}}{e_m} \left[\widetilde{v}^{m,Q}_j  \left( \frac{\widetilde{v}^{m,Q-1}_i - \widetilde{v}^{m,Q}_i }{\Delta y}\right)   - \widetilde{v}^{m,Q}_i \left( \frac{ \widetilde{v}^{m,Q-1}_j -\widetilde{v}^{m,Q}_j}{\Delta y}\right)  \right] &  = -e^d_m \widetilde{v}_i^{m,Q} + \widehat{\phi}_i^m.\\ \nonumber
\end{align}
The nonlinear system of equations (\ref{eq:bulk})-(\ref{eq:BC})-(\ref{eq:BC1}), whose unknowns are $\left(\widetilde{v}_i^{m,q}\right)_{0\leq i \leq n, 1\leq q \leq Q}$ 
is solved using Newton iterations with initial guess $\left(v_i^{m-1,q}\right)_{0\leq i \leq n, 1\leq q \leq Q}$. The obtained solution does not satisfy in general 
the desired non-negativeness and volumic constraints. 
This is the reason why an additional projection step is performed. For all $0\leq i \leq n$ and $1\leq q\leq Q$, we define
$$
v_i^{m,q}:= \frac{[\widetilde{v}_i^{m,q}]_+}{\sum_{j=0}^n[\widetilde{v}_j^{m,q}]_+}, 
$$
so that 
$$
v_i^{m,q}\geq 0 \quad \mbox { and } \sum_{j=0}^n v_j^{m,q} = 1.
$$

We numerically observe that this scheme is unconditionally stable with respect to the choice of discretization parameters $\Delta t$ and $\Delta y$. 

\medskip 
A standard practice in the production of thin film CIGS (Copper, Indium, Gallium, Selenium) solar cells by means of PVD process is to consider piecewise-constant external fluxes. We refer the reader to~\cite{theseAthmane} for further details. 
In the following numerical tests, we consider time-dependent functions of the form  
\begin{equation} 
\label{def:piecewiseconstant}
 \phi_{\rm i}  (t) = \left \{  
\begin{array}{lcl} 
\alpha^i_1 & \quad & 0<t\leq \tau^i_1, \\
\alpha^i_2 & \quad & \tau^i_1<t\leq \tau^i_2, \\
\alpha^i_3 & \quad & \tau^i_2<t\leq T, \\
\end{array} 
\right. 
\end{equation} 
where $0<\tau^i_1 < \tau^i_2 <T$ and $(\alpha^i_1, \alpha^i_2, \alpha^i_3) \in (\R_+)^3$ are non-negative constants for all $0 \leq i \leq n$. 
Besides, we consider initial condition of the form 
\begin{equation} 
\label{def:initialcondition}
v^0_i(y) = \dfrac{w_i(y)}{\sum_{j=0}^n w_j(y)} \quad \forall 0 \leq i \leq n, 
\end{equation} 
where $w_i:[0,1] \to \R_+$ are functions which will be precised below.  In the whole section, system~\eqref{eq:rescaled} is simulated with four species (i.e. $n=3$).  
\medskip 

In Figure~\ref{fig:simple_simulation}  are plotted the results obtained for the simulation of~\eqref{eq:rescaled} with the following parameters : 
\begin{itemize} 
\item $T=200$, $M=200$, $Q=100$, $\Delta t = 1$, $\Delta y = 0.01$, $e_0 = 1$.  
\item Cross-diffusion coefficients $K_{ij}$ 

$$
\centering
\begin{array}{|c|c|c|c|c| }
 \hline
    	&j=0 		&  j=1 		&  j=2		& j=3 \\ \hline
   i=0 & 0	 	& 0.1141 		& 0.0776 		&  0.0905 \\
   i=1 & 0.1141	& 0 			&  0.0646	 	&0.0905\\
   i=2 & 0.0776 	& 0.0646 		& 0	 		&0.0905\\ 
   i=3 & 0.0905 	& 0.0905		& 0.0905 		&0\\ 
   \hline
\end{array}
$$

\item External fluxes of the form~\eqref{def:piecewiseconstant} with $\tau_1^i = 66$ and $\tau^i_2 = 132$ for every $0\leq i \leq n$ and with
$$
\centering
\begin{array}{|c|c|c|c|c| }
 \hline
    &i=0&  i=1 & i=2 &  i=3 \\ \hline
   \alpha^i_1 & 0.9 & 2 & 0.2 & 0.7 \\
   \alpha^i_2 & 1.4 & 1.5 & 1.2 &0.3\\
   \alpha^i_3 & 0.9 & 2 & 0.2 &0.7\\ \hline
\end{array}
$$
\item Initial concentrations $v^0_i$ of the form~\eqref{def:initialcondition} with $w_0(y) = y$, $w_1(y) = 2y$, $w_2(y) = \sqrt{y}$ and $w_3(y) =0$.
\end{itemize}   

The profile of the external fluxes is plotted in Figure~\ref{fig:simple_simulation}-(a). In Figure~\ref{fig:simple_simulation}-(b) and Figure~\ref{fig:simple_simulation}-(c)  are given respectively the the initial and the final concentrations of the four species.

%%% FIGURE %%%%%
\begin{figure}[h!] 
\centering 
\subfloat[]{\includegraphics[width=5.2cm, height=4.5cm]{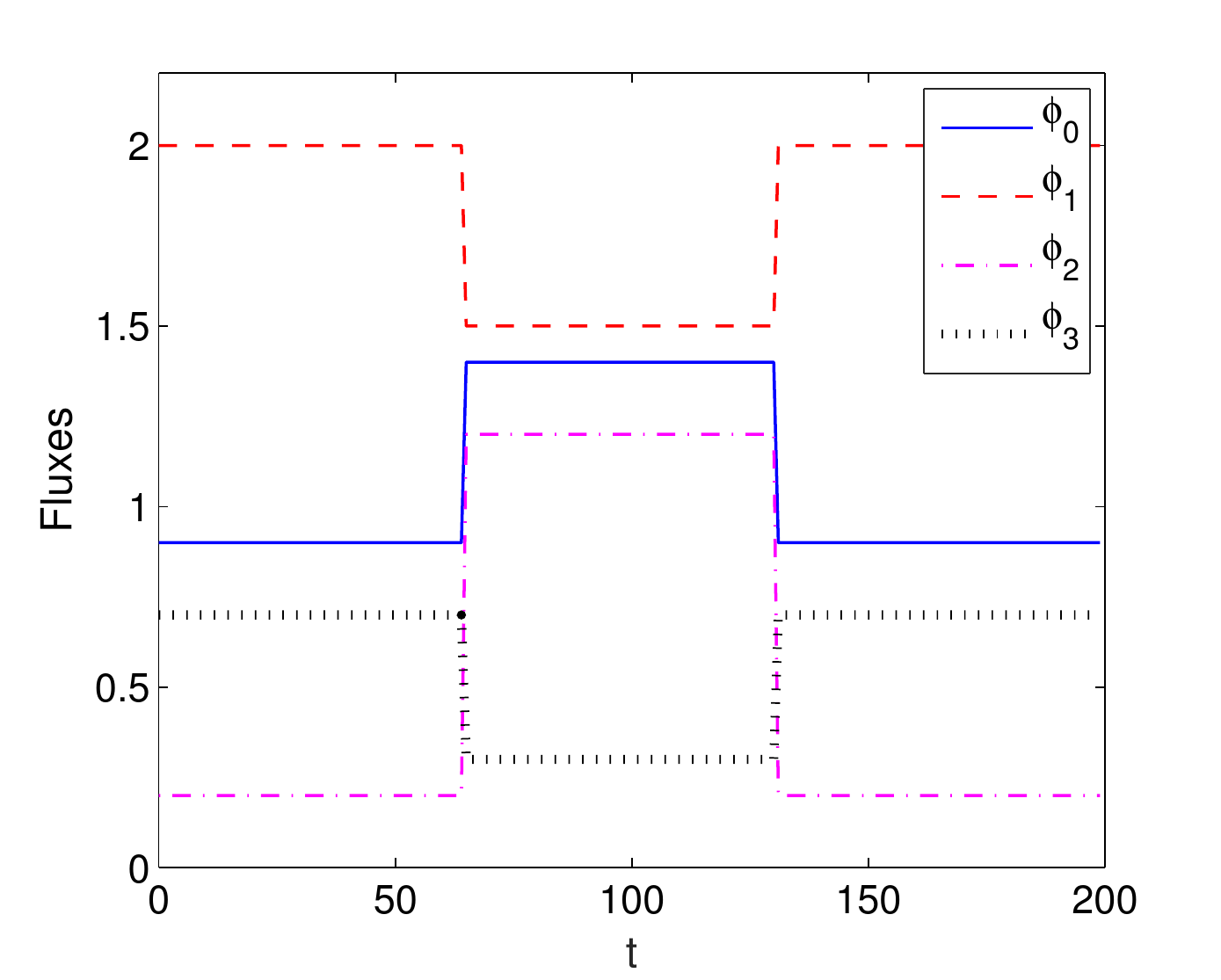}}
\subfloat[]{\includegraphics[width=5.2cm, height=4.5cm]{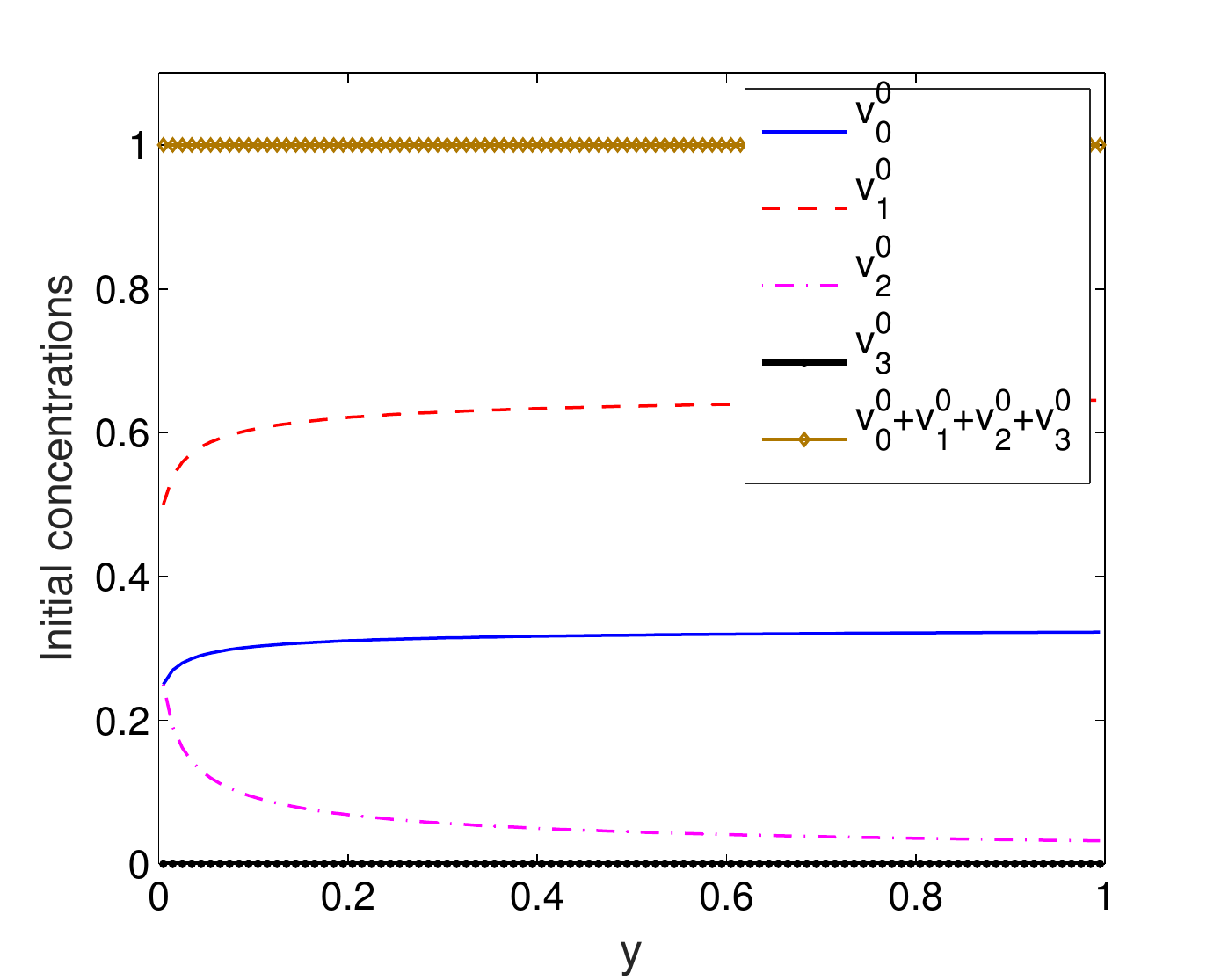}}
\subfloat[]{\includegraphics[width=5.2cm, height=4.5cm]{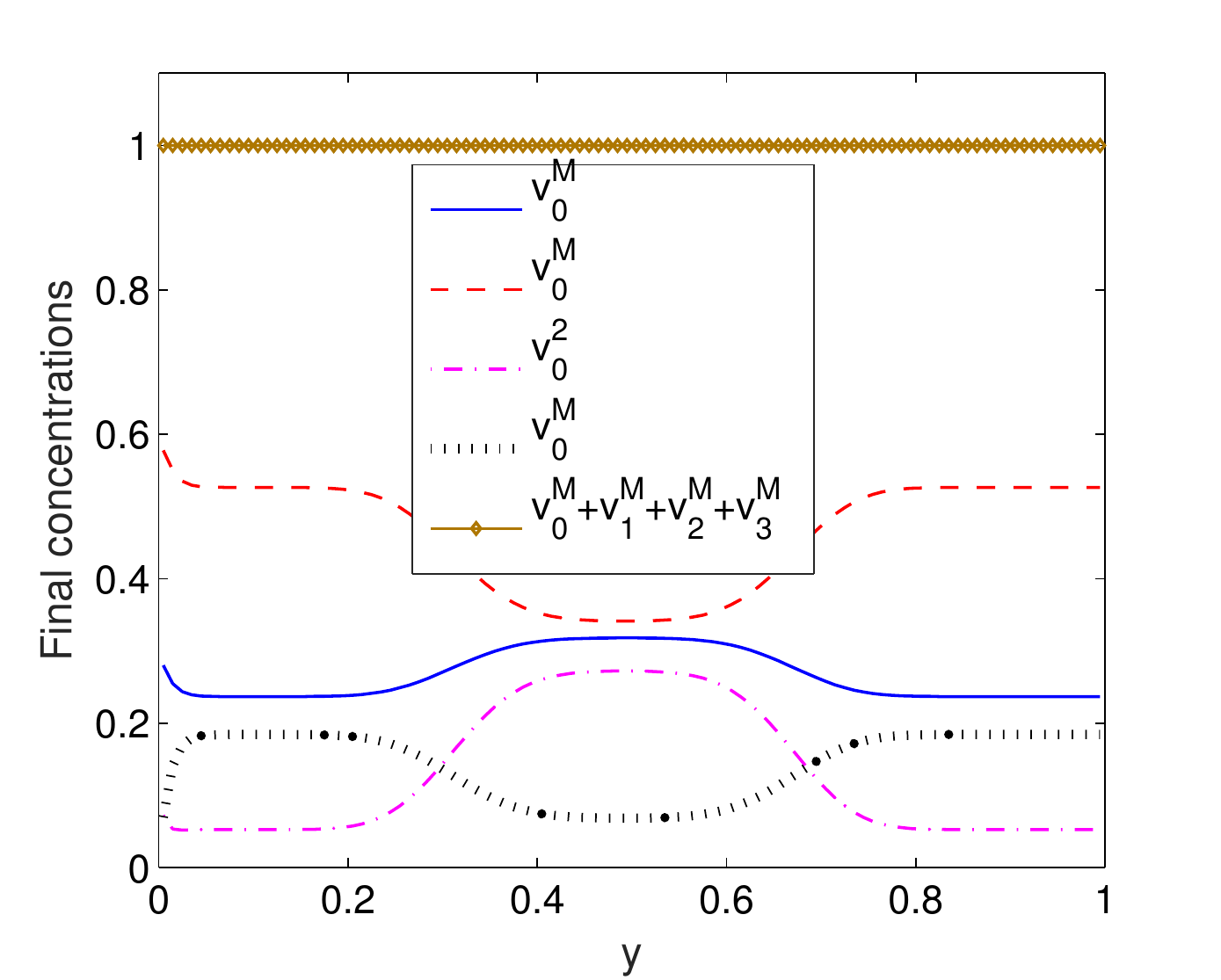}}
\caption{Simulation of \eqref{eq:rescaled}.} 
\label{fig:simple_simulation} 
\end{figure} 
%%%%%%%%%%%

\subsection{Long-time behaviour results}\label{sec:longtimeres}
In this section is given a numerical illustration of Proposition~\ref{prop:longtime}. We consider time-dependent functions of the form  
\begin{equation} 
\label{def:constant}
\phi_i(t) = \beta_i, \quad \forall 0 \leq t \leq T. 
\end{equation}
where $(\beta_i)_{0 \leq i \leq n} \in (\R^*_+)^{n+1}$. In Figure~\ref{fig:longtime} are plotted the results obtained for the the simulation of~\eqref{eq:rescaled} with the following parameters : 
\begin{itemize} 
\item   {$T=2000$, $M=2000$, $Q=100$, $\Delta t = 1$, $\Delta y = 0.01$, $e_0 = 1$.}  
\item Cross-diffusion coefficients $K_{ij}$ 

$$
\centering
\begin{array}{|c|c|c|c|c| }
 \hline
    	&j=0 		&  j=1 		&  j=2		& j=3 \\ \hline
   i=0 & 0	 	& 0.1141 		& 0.0776 		&  0.0905 \\
   i=1 & 0.1141	& 0 			&  0.0646	 	&0.0905\\
   i=2 & 0.0776 	& 0.0646 		& 0	 		&0.0905\\ 
   i=3 & 0.0905 	& 0.0905		& 0.0905 		&0\\ 
   \hline
\end{array}
$$

\item External fluxes of the form~\eqref{def:constant} with 
$$
\centering
\begin{array}{|c|c|c|c|c| }
 \hline
    &\mbox{i=0} &   \mbox{i=1} &   \mbox{i=2} &   \mbox{i=3} \\ \hline
   \beta^i & 0.9 &  {0.8} & 1.7 & 0.5 \\ \hline
\end{array}
$$
\item Initial concentrations $v^0_i$ of the form~\eqref{def:initialcondition} with 
$$
 {w_0(y) = \exp \left(-\dfrac{(y-0.5)^2}{0.04} \right), \quad w_1(y) = y^2, \quad  w_2(y) = 1 - w_0(y), \quad w_3(y) = | \sin(\pi y) |}.$$
\end{itemize}   

For all $0\leq i \leq n$, let $\bar{v}_i := \beta_i / \sum_{j=0}^n \beta_j$.  {We consider the time-dependent quantity 
$$
\gamma(t) = \dfrac{1}{\overline{h}(v(t,\cdot))}
$$
where the relative entropy $\overline{h}$ is defined in \eqref{eq:relentropy}. We also consider the quantities 
$$\eta_i(t) = \dfrac{1}{  \| v_i(t, \cdot) - \bar{v}_i\|^2_{L^1(0,1)}}$$ 
and 
$$\eta(t) =\dfrac{1}{  \sum \limits_{i=0}^n   \| v_i(t, \cdot) - \bar{v}_i\|^2_{L^1(0,1)}}   $$
}

In Figure~\ref{fig:longtime}-(a) and~\ref{fig:longtime}-(b) are plotted respectively the initial and the final concentration profiles.

\medskip 
  {The evolution of $\left(\eta_i(t) \right)_{0 \leq i \leq n}$ (respectively $\eta(t)$ and $\gamma(t)$) with respect to $t$ is shown in Figure~\ref{fig:longtime}-(c) (respectively \ref{fig:longtime}-(d) and ~\ref{fig:longtime}-(e)). We numerically observe that these quantities are affine functions of $t$ in the asymptotic regime 
 which illustrates the theoretical result of Proposition~\ref{prop:longtime}.}

%%% FIGURE %%%%%
\begin{figure}[h!]  
\centering
\subfloat[]{
\includegraphics[width=7cm, height=6cm]{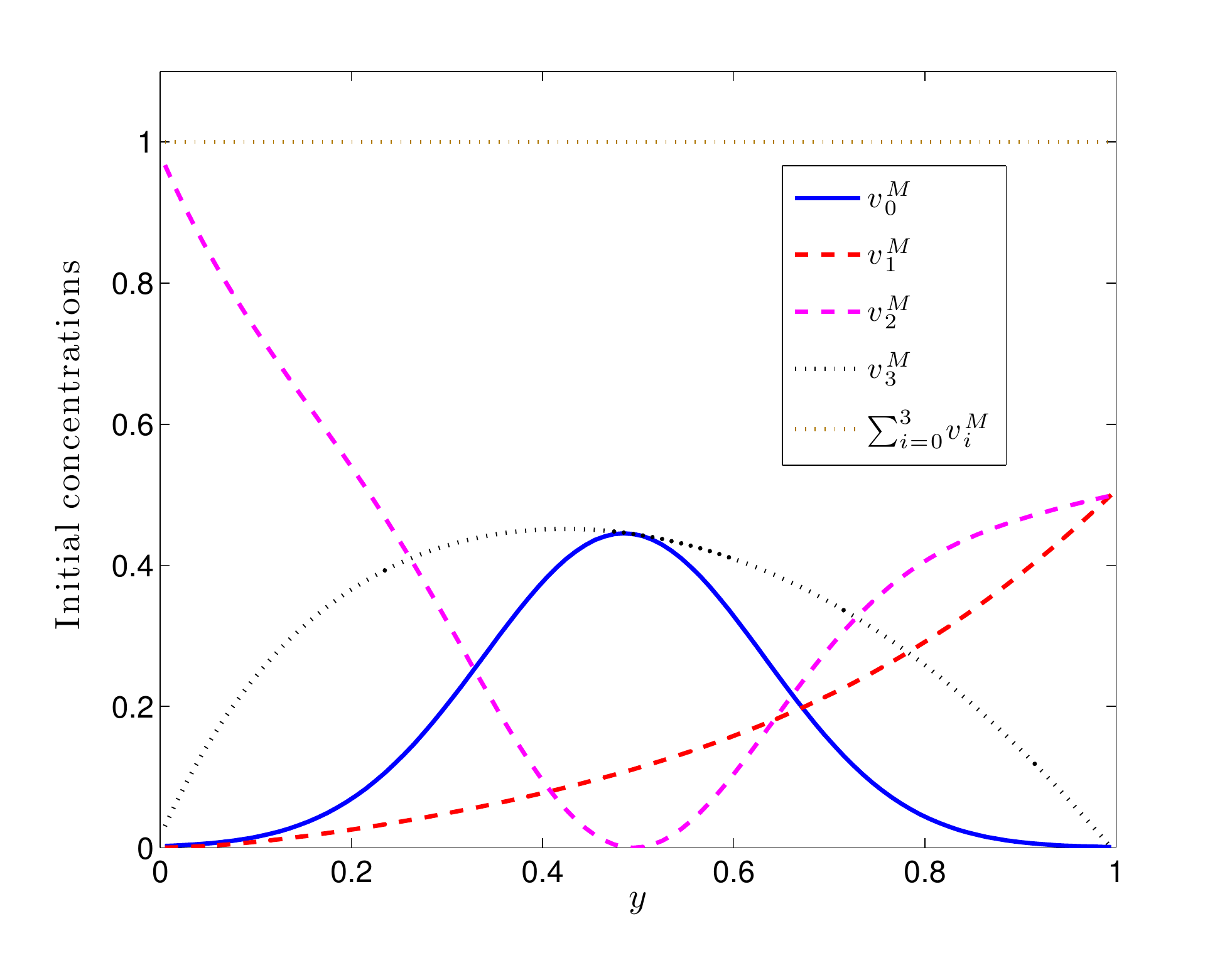}}
\subfloat[]{
\includegraphics[width=7cm, height=6cm]{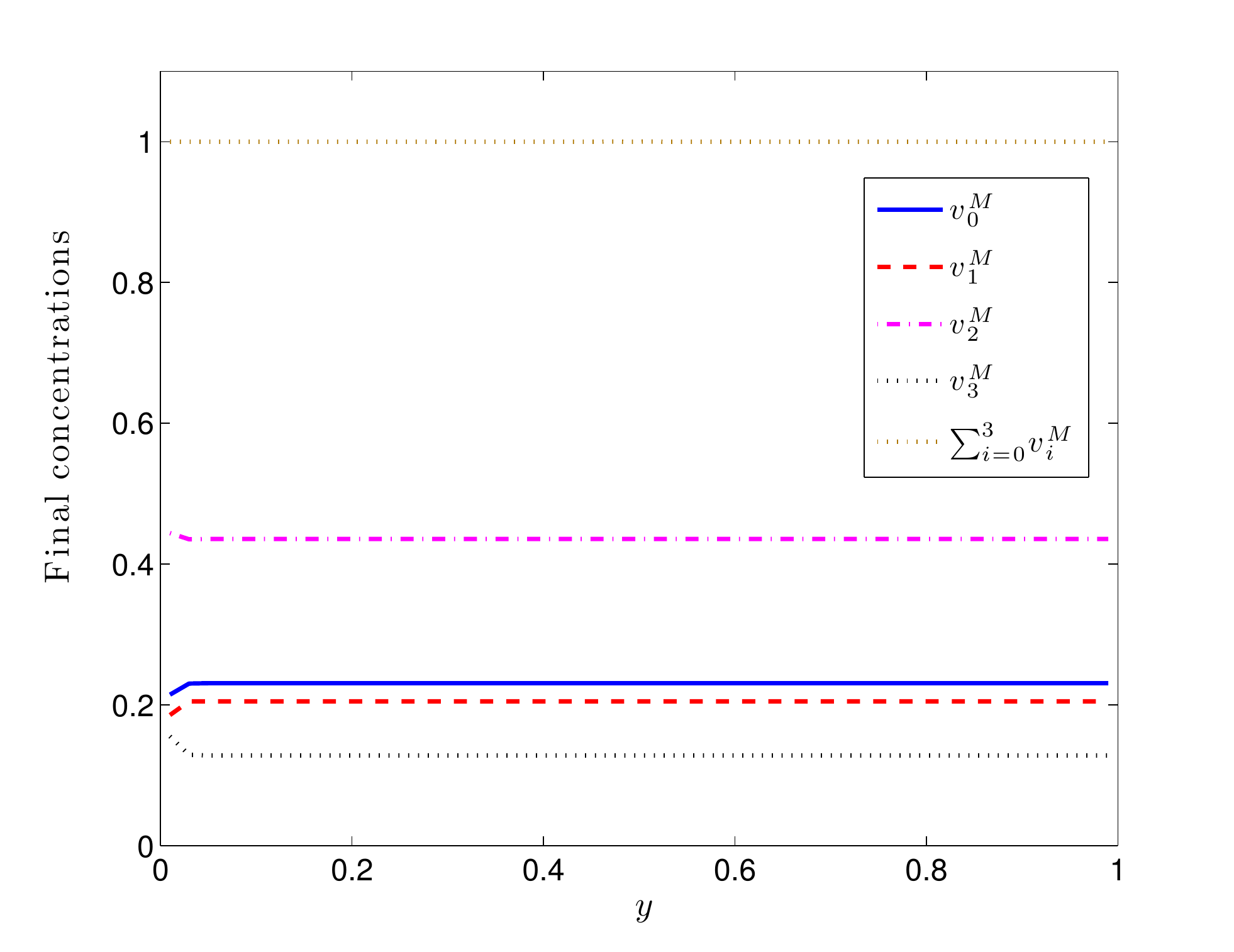}}
\newline
\subfloat[]{
\includegraphics[width=7cm, height=6cm]{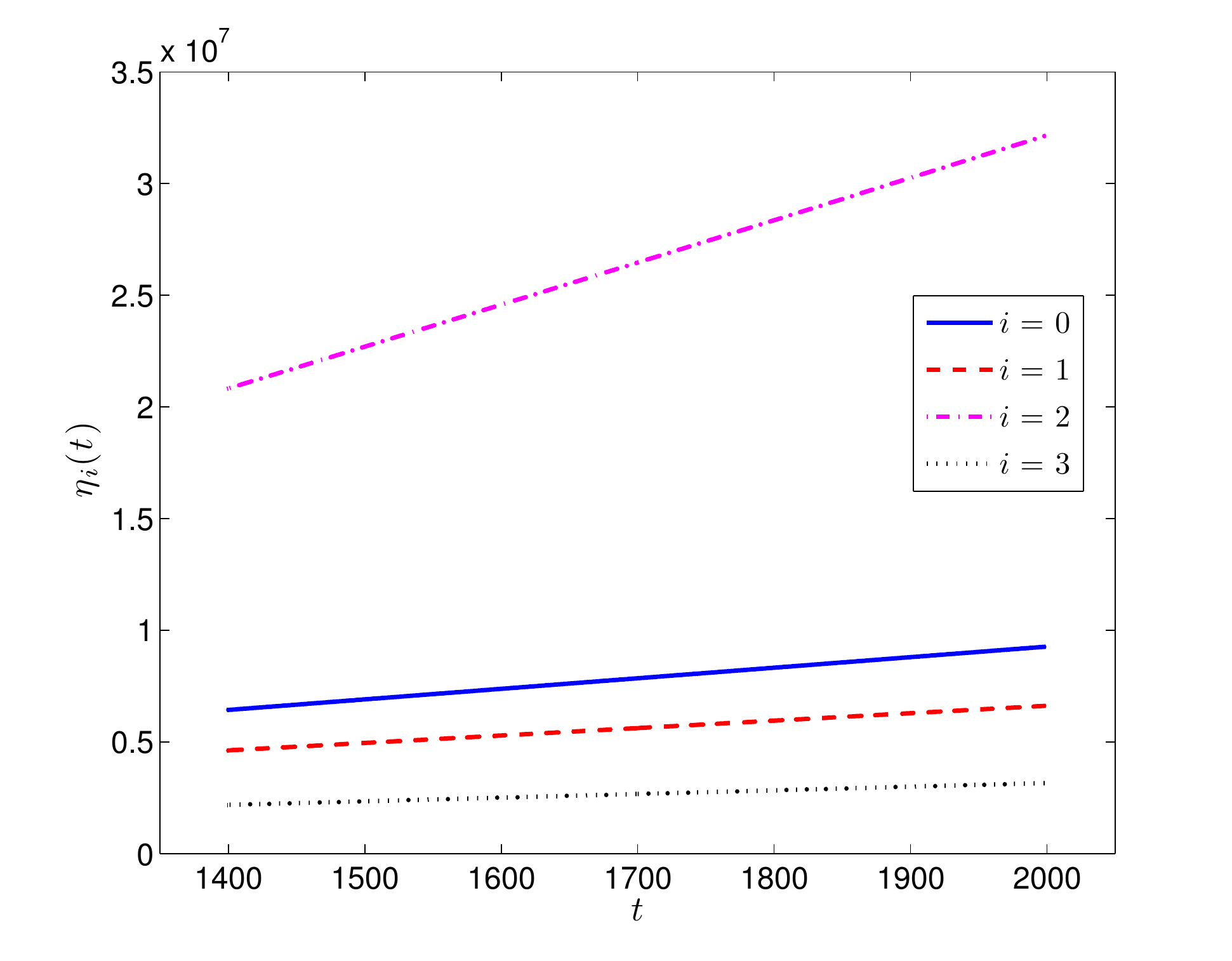}}
\subfloat[]{
\includegraphics[width=7cm, height=6cm]{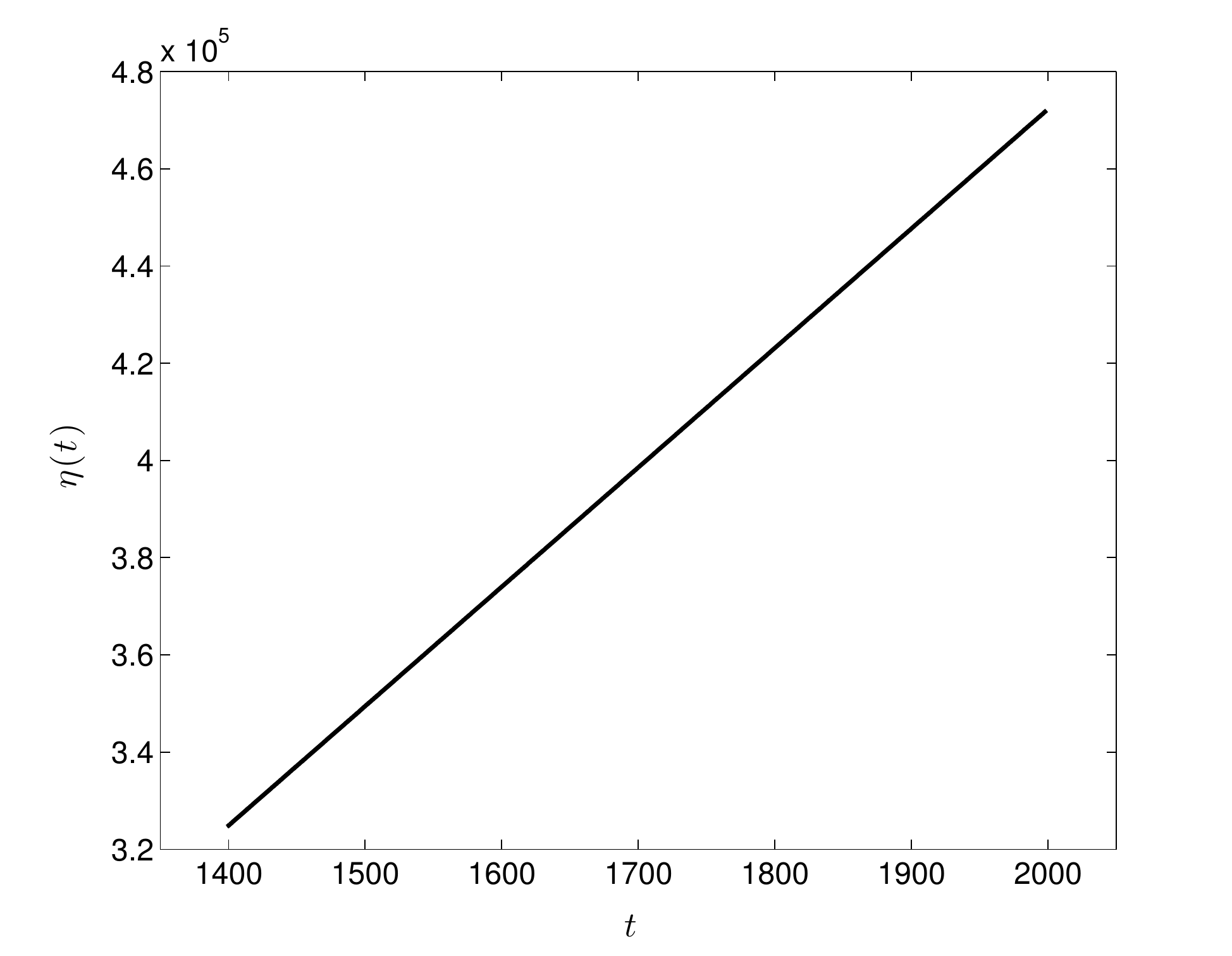}}
\newline
\subfloat[]{
\includegraphics[width=7cm, height=6cm]{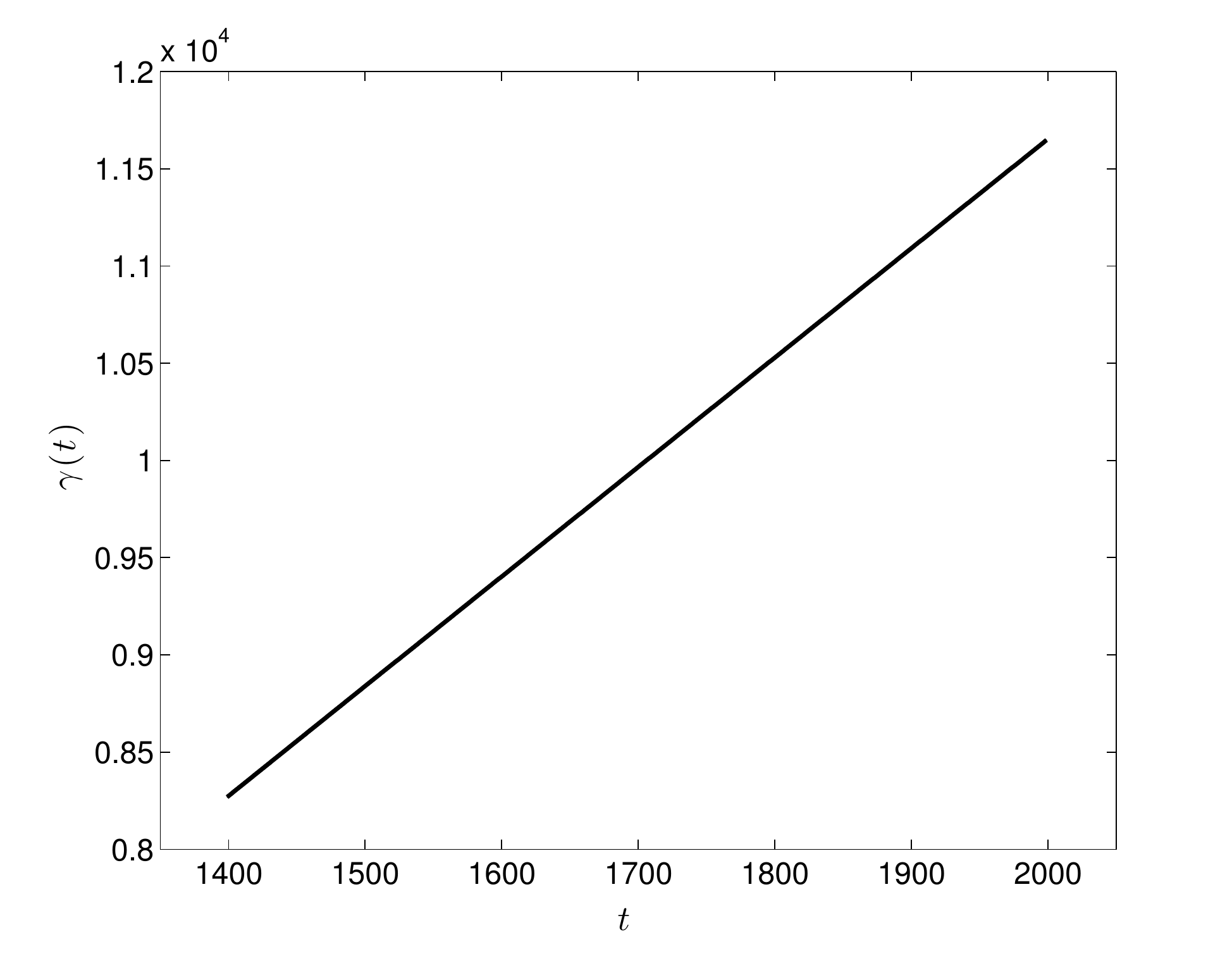}}
\caption{Long-time behavior in the case of non negative constant external fluxes.}
\label{fig:longtime}
\end{figure} 

\subsection{Optimization of the fluxes}\label{sec:optimalres}

The optimization problem~\eqref{eq:optimal} is solved in practice using an adjoint formulation associated to the discretization scheme described in Section~\ref{sec:scheme}. We refer 
the reader to~\cite{theseAthmane} for more details and comparisons between our model and experimental results obtained in the context of thin film CIGS solar cell fabrication. To illustrate Proposition~\ref{prop:optimal}, we proceed as follows: first, we perform a simulation of~\eqref{eq:rescaled} with external fluxes $\Phi_{\rm sim}$ for a duration $T$ to obtain a final thickness $e_{\Phi_{\rm sim}}(T)$ and final concentrations $v_{\Phi_{\rm sim}}(T, \cdot)$, then, we solve the minimization problem~\eqref{eq:optimal} to obtain optimal fluxes $\Phi^*$ where the target concentrations are 
$$v_{\rm opt} (y) = v_{\Phi_{\rm sim}}(T, y) \quad \forall ~y \in (0,1) $$
and the target thickness is 
$$ e_{\rm opt} = e_{\Phi_{\rm sim}}(T).$$ 
Lastly, we perform another simulation of~\eqref{eq:rescaled} with the obtained optimal fluxes $\Phi^*$ and compare the final concentrations $v_{\Phi^*}$ and the final thickness $e_{\Phi^*}$ to the target concentrations $v_{\rm opt}$ and the target thickness $ e_{\rm opt}$.

In Figures~\ref{fig:species0}-(a), \ref{fig:species1}-(a), \ref{fig:species2}-(a) and \ref{fig:species3}-(a) are plotted the final concentration profiles $v_{\Phi_{\rm sim}}(T, \cdot)$ resulting from the simulation of~\eqref{eq:rescaled} with the following parameters :
\begin{itemize} 
\item $T=120$, $M=120$, $Q=100$, $\Delta t = 1$, $\Delta y = 0.01$, $e_0 = 1$.  
\item Cross-diffusion coefficients $K_{ij}$ 
$$
\centering
\begin{array}{|c|c|c|c|c| }
 \hline
    	&j=0 		&  j=1 		&  j=2		& j=3 \\ \hline
   i=0 & 0	 	& 0.1141 		& 0.0776 		&  0.0905 \\
   i=1 & 0.1141	& 0 			&  0.0646	 	&0.0905\\
   i=2 & 0.0776 	& 0.0646 		& 0	 		&0.0905\\ 
   i=3 & 0.0905 	& 0.0905		& 0.0905 		&0\\ 
   \hline
\end{array}
$$

\item External fluxes $\Phi_{\rm sim}$ of the form~\eqref{def:piecewiseconstant} with 
$$
\centering
\begin{array}{|c|c|c|c|c| }
 \hline
    &i=0&  i=1 & i=2 &  i=3 \\ \hline
   \alpha^i_1 & 0.9 & 2 & 0.2 & 0.7 \\
   \alpha^i_2 & 1.4 & 1.5 & 1.2 &0.3\\
   \alpha^i_3 & 0.9 & 2 & 0.2 &0.7\\ \hline
\end{array}
$$

\item Initial concentrations $v^0_i$ of the form~\eqref{def:initialcondition} with $w_0(y) = y$, $w_1(y) = 2y$, $w_2(y) = \sqrt{y}$ and $w_3(y) =0$.
\end{itemize}   
We use a quasi-Newton iterative gradient algorithm for the resolution of the minimization problem. At each iteration of the algorithm, the approximate hessian is updated by means of a BFGS procedure and the optimal step size is the solution of a line search subproblem. More details on the numerical optimization algorithms can be found in \cite{BONNANS}.  The initial guess $\Phi^0$ is taken of the form~\eqref{def:constant} where $\beta^i = 1$ for all $0 \leq i \leq n$. 

\medskip 
The algorithm is run until one of the following stopping criterion is reached : either $ \left(  \cJ(\Phi) \leq \varepsilon \right) $ or $\left( \| \nabla_{\Phi}\cJ(\Phi) \|_{L^2} \leq \nu \right)$ with $\varepsilon = 10^{-5}$ and $\nu = 10^{-5}$.

\begin{figure}[h!]  
\centering
\subfloat[]{\includegraphics[width=7cm, height=5cm]{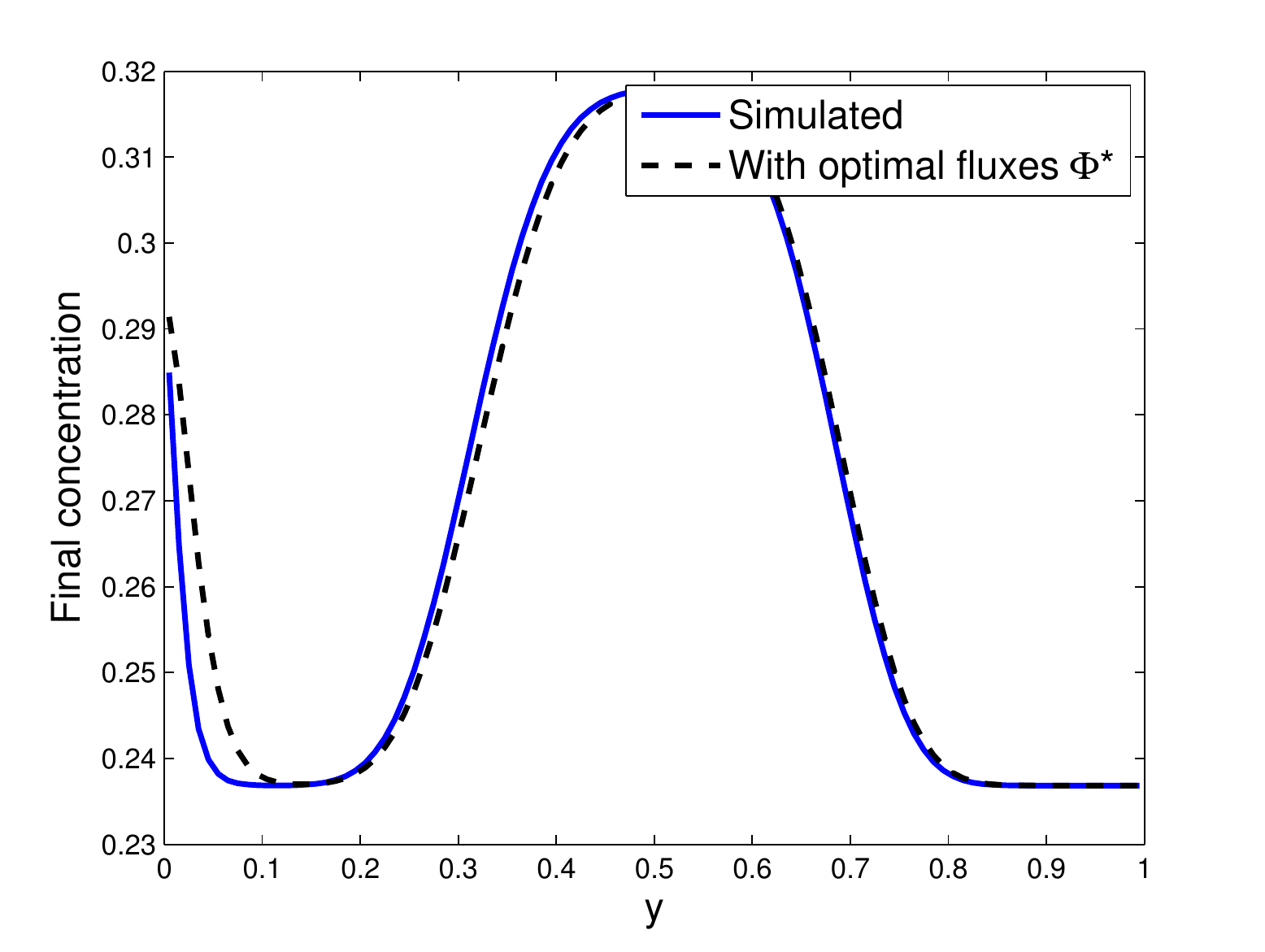}}
\subfloat[]{\includegraphics[width=7cm, height=5cm]{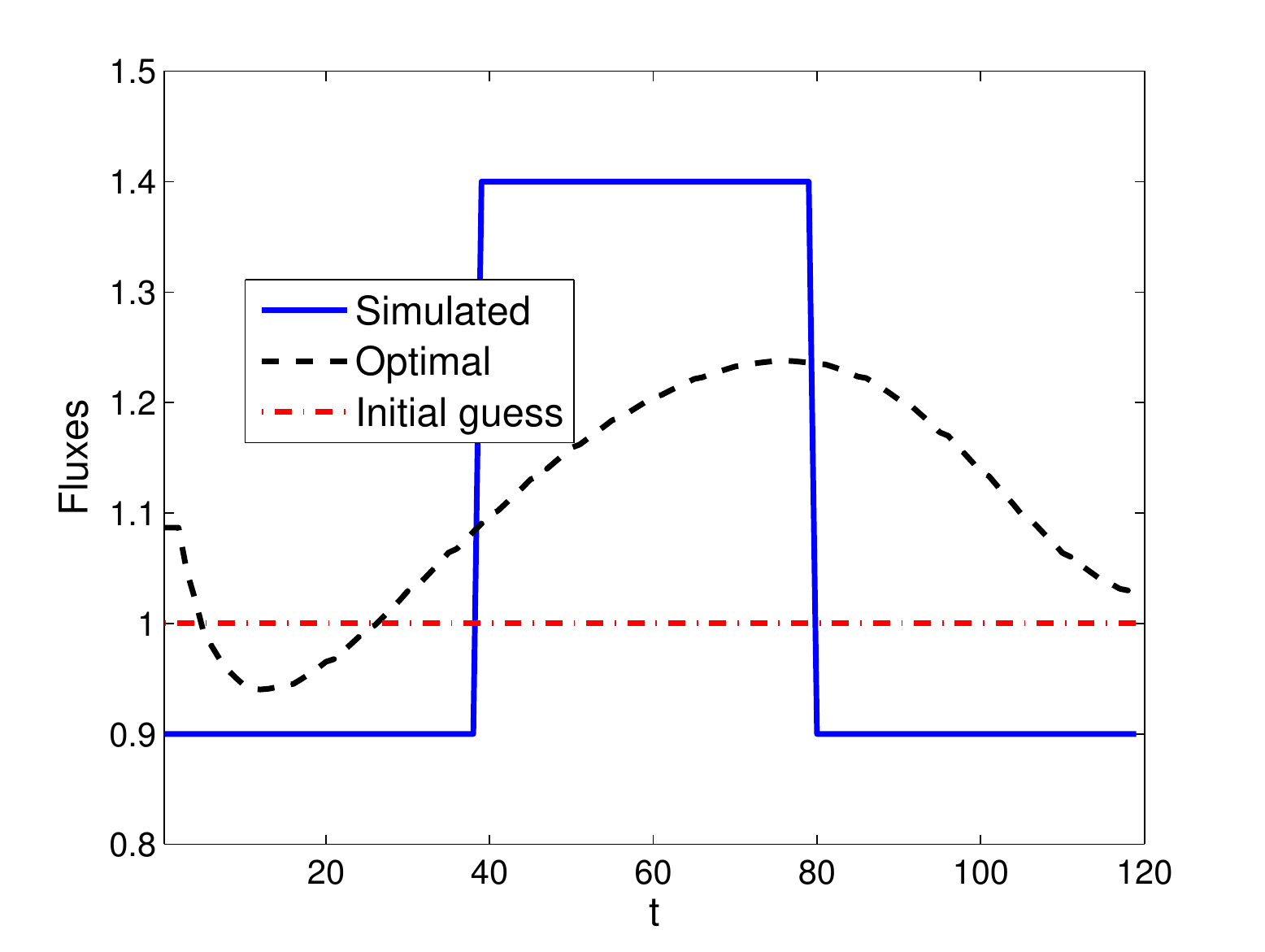}} 
\caption{Reconstruction of the final concentration of the species $i=0$.}
\label{fig:species0}
\end{figure} 
\begin{figure}[h!]  
\centering
\subfloat[]{\includegraphics[width=7cm, height=5cm]{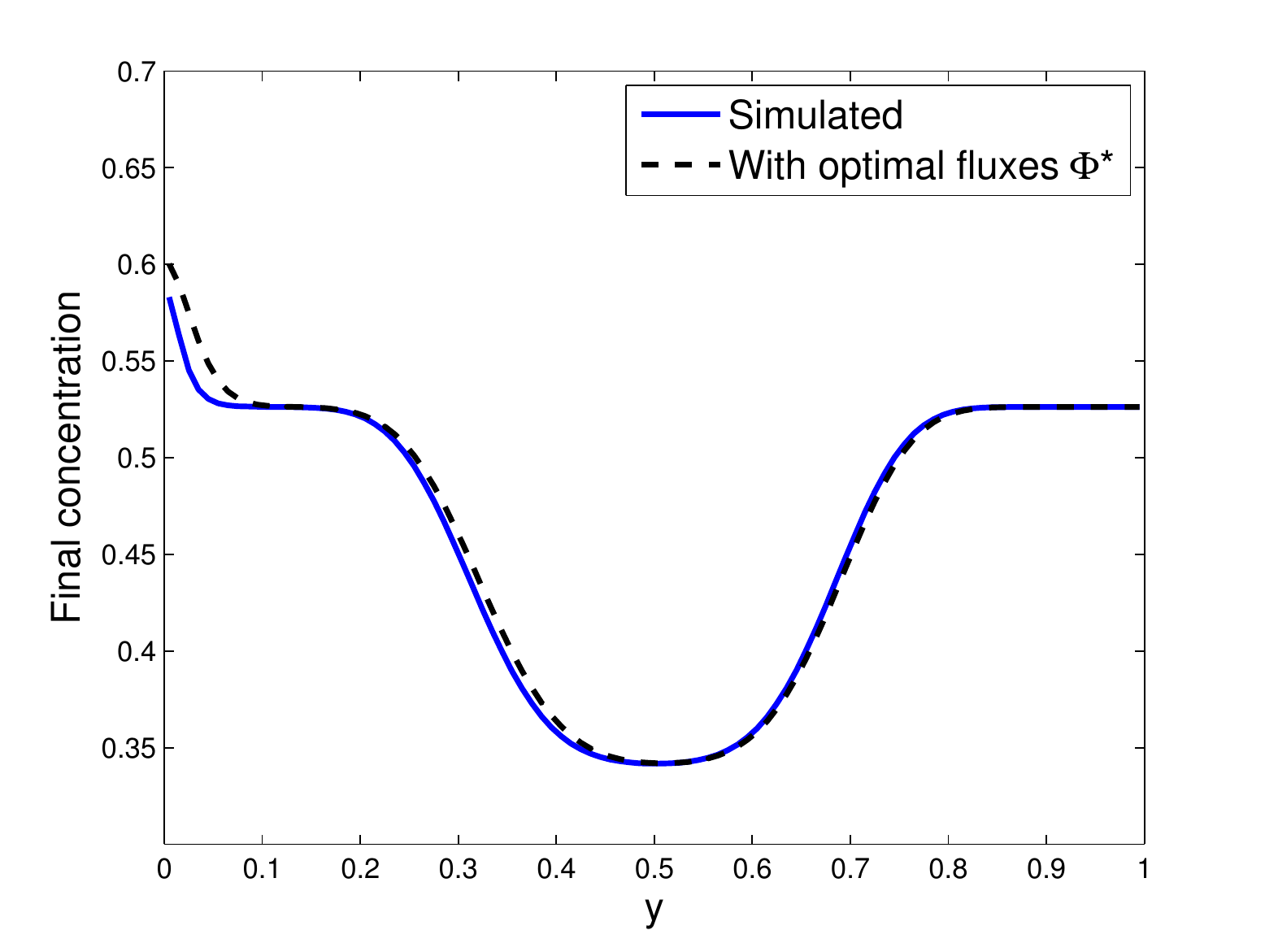}}
\subfloat[]{\includegraphics[width=7cm, height=5cm]{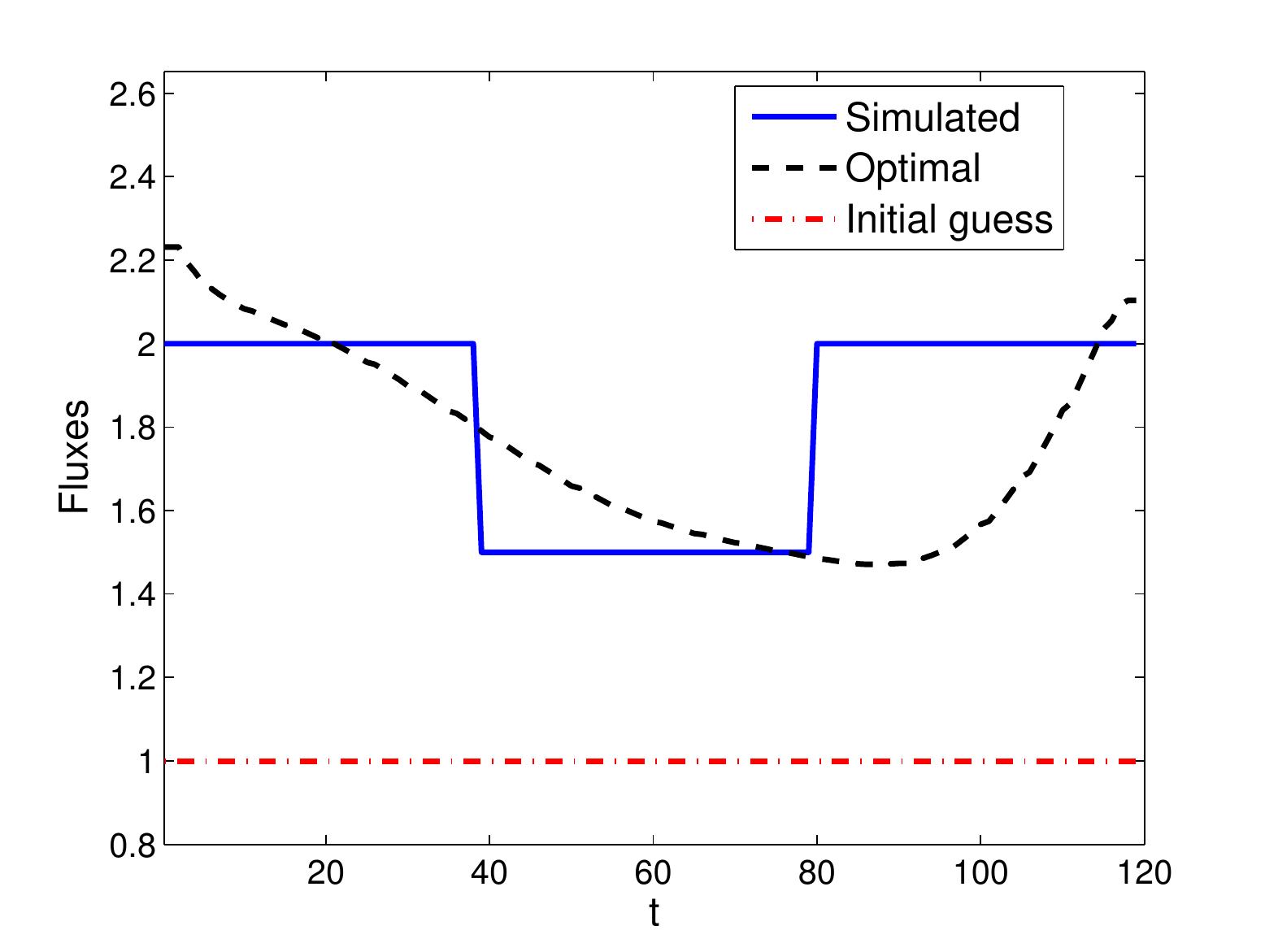}} 
\caption{
Reconstruction of the final concentration of the species $i=1$. 
}
\label{fig:species1}
\end{figure} 

\begin{figure}[h!]  
\centering
\subfloat[]{\includegraphics[width=7cm, height=5cm]{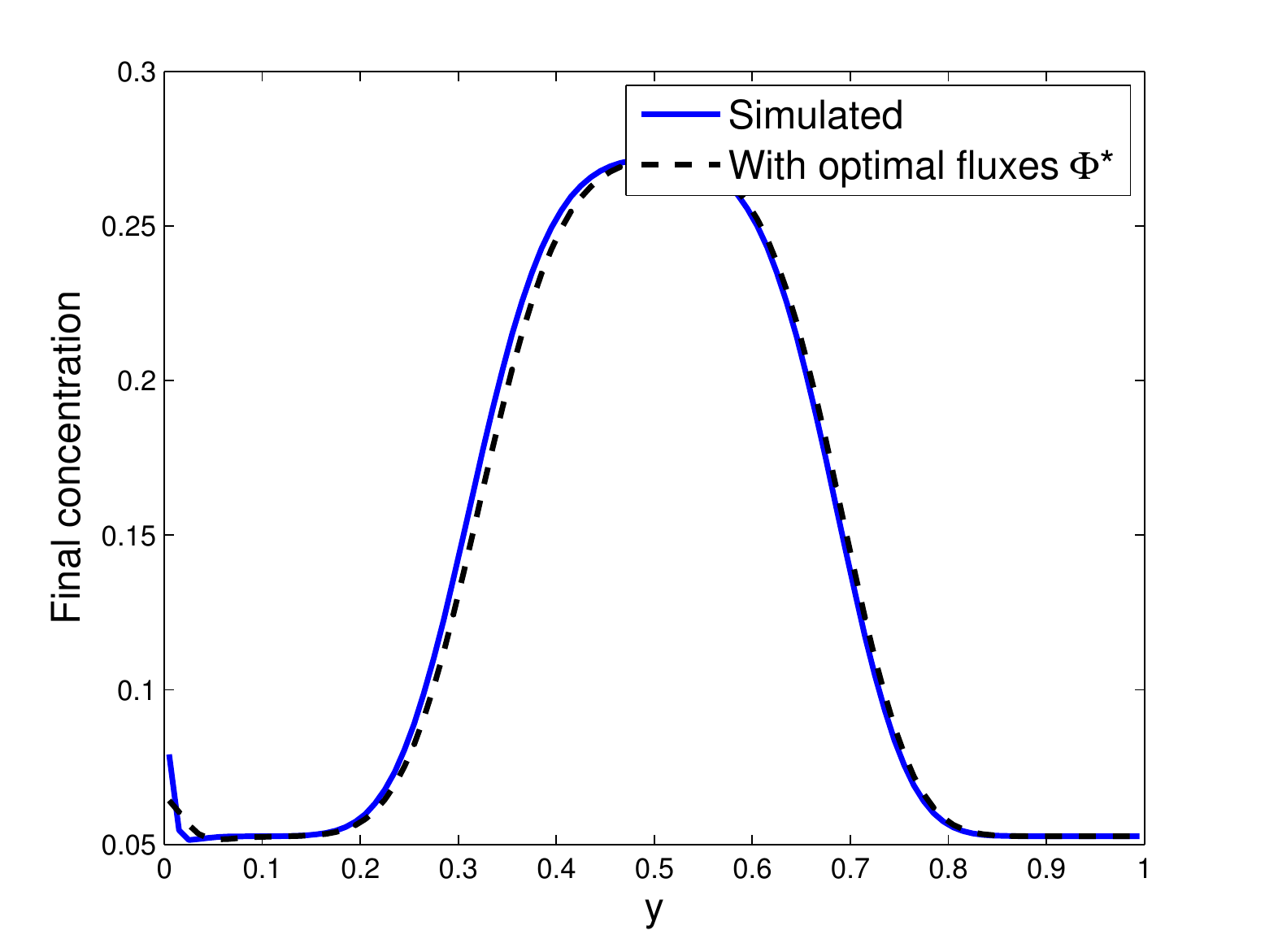}}
\subfloat[]{\includegraphics[width=7cm, height=5cm]{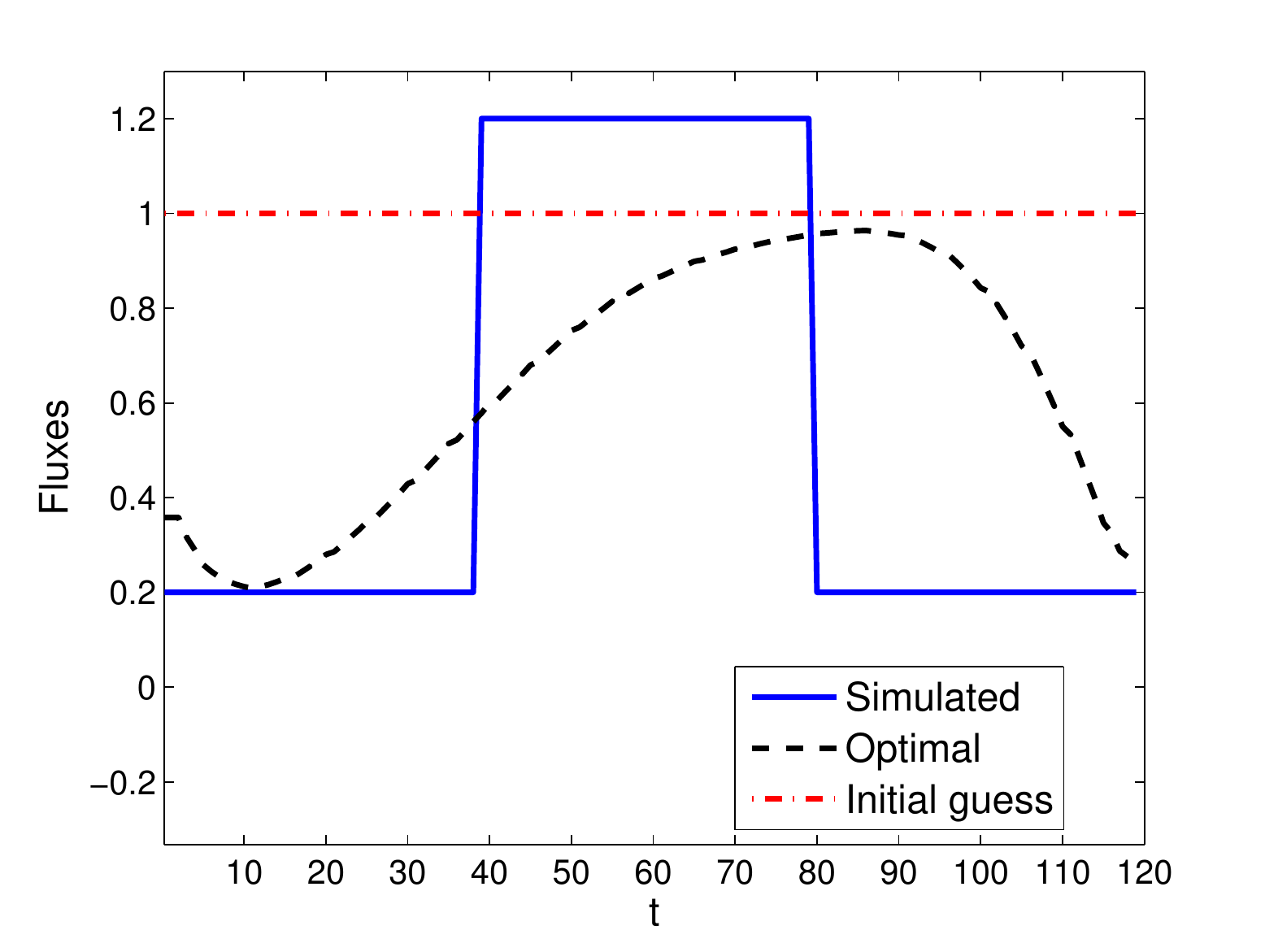}} 
\caption{
Reconstruction of the final concentration of the species $i=2$. 
}

\label{fig:species2}
\end{figure} 

\begin{figure}[h!]  
\centering
\subfloat[]{\includegraphics[width=7cm, height=5cm]{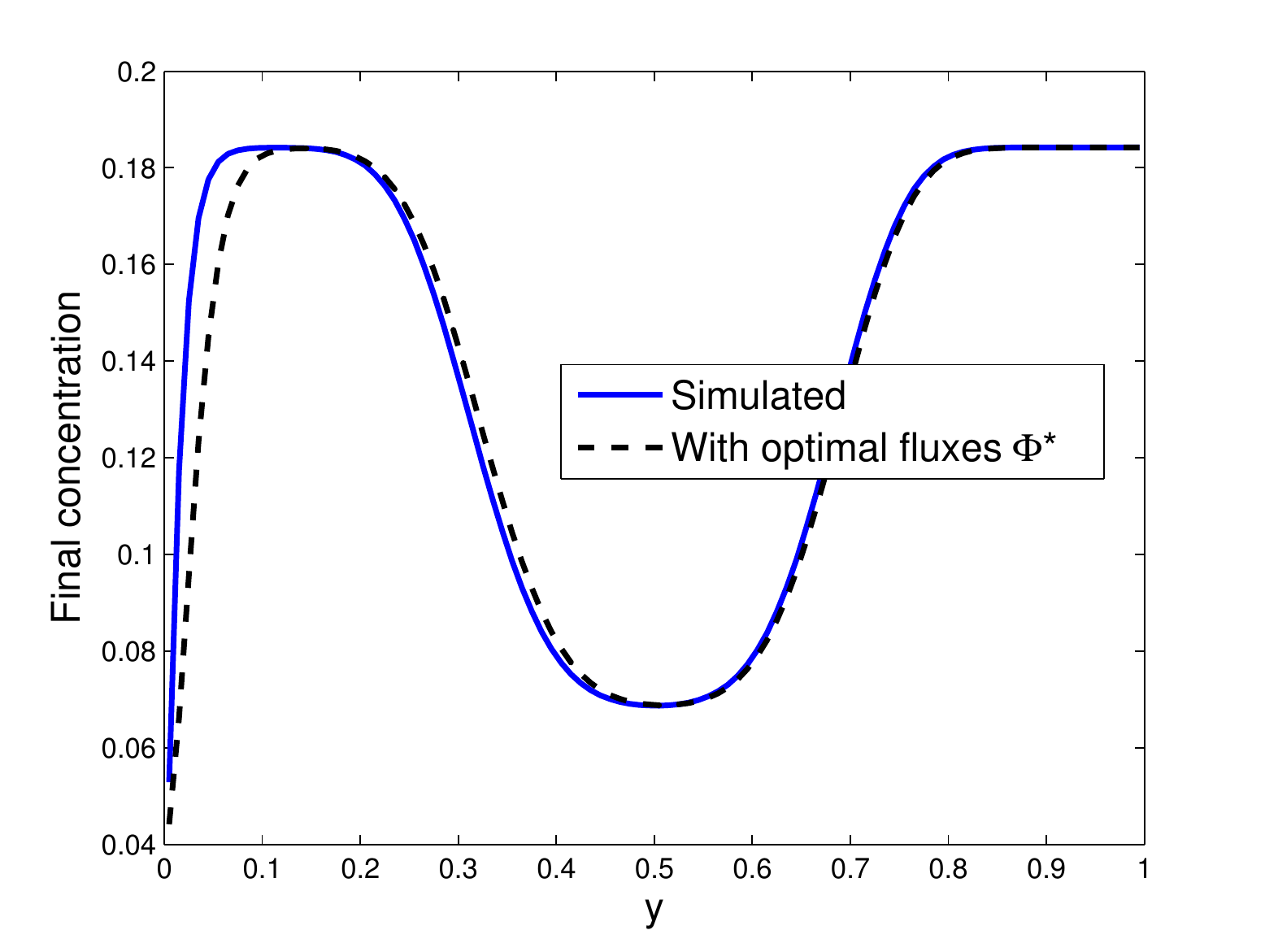}}
\subfloat[]{\includegraphics[width=7cm, height=5cm]{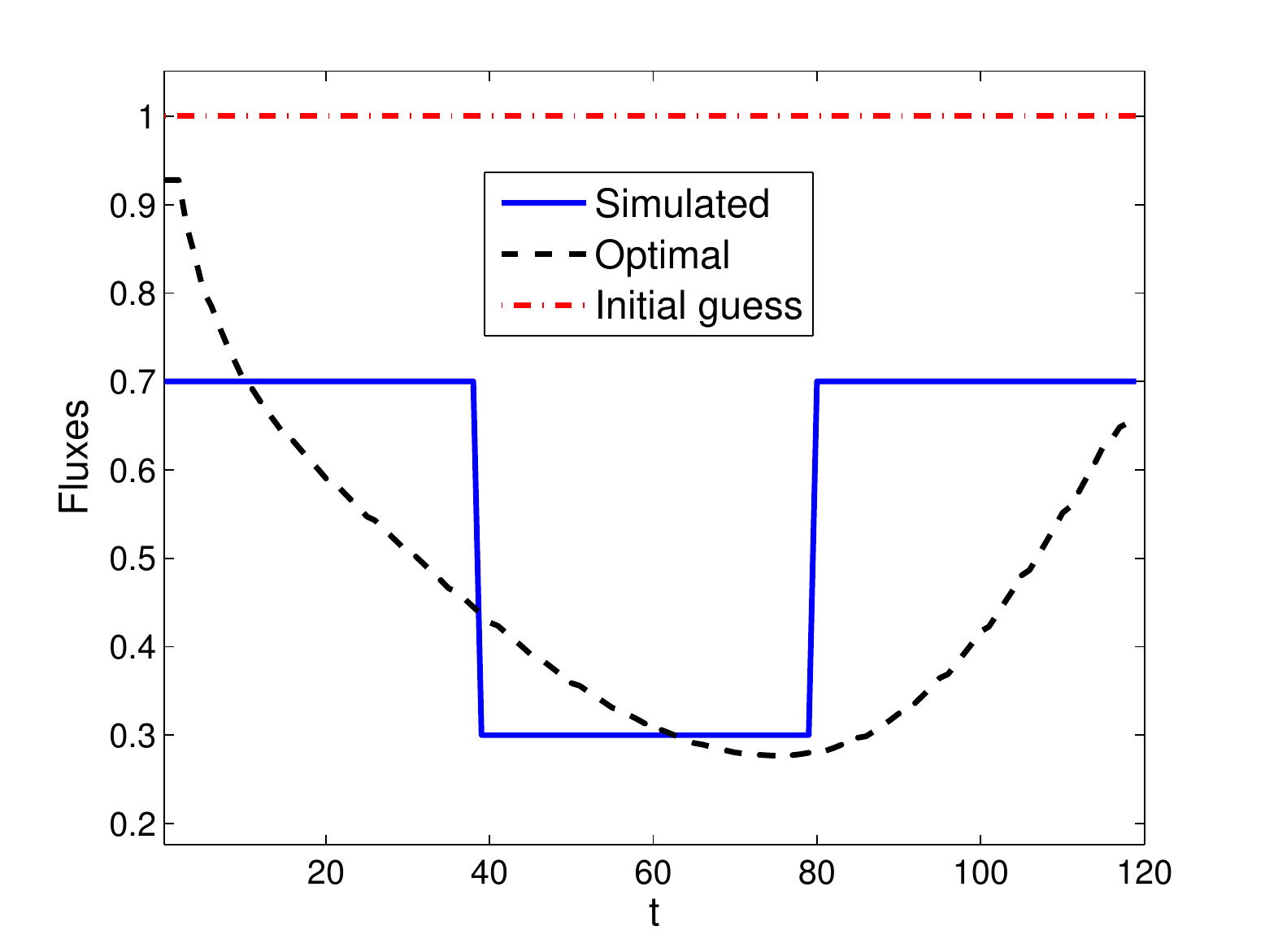}}
\caption{
Reconstruction of the final concentration of the species $i=3$. 
}
\label{fig:species3}
\end{figure} 

In Figure~\ref{fig:CONV}-(a) we plot the evolution of the value of the cost $\cJ(\Phi)$ with respect to the number of iterations. We refer the reader to \cite{theseAthmane} for more details and comparison between different minimization approaches.  

\begin{figure}[h!]  
\centering
\includegraphics[width=8cm, height=6cm]{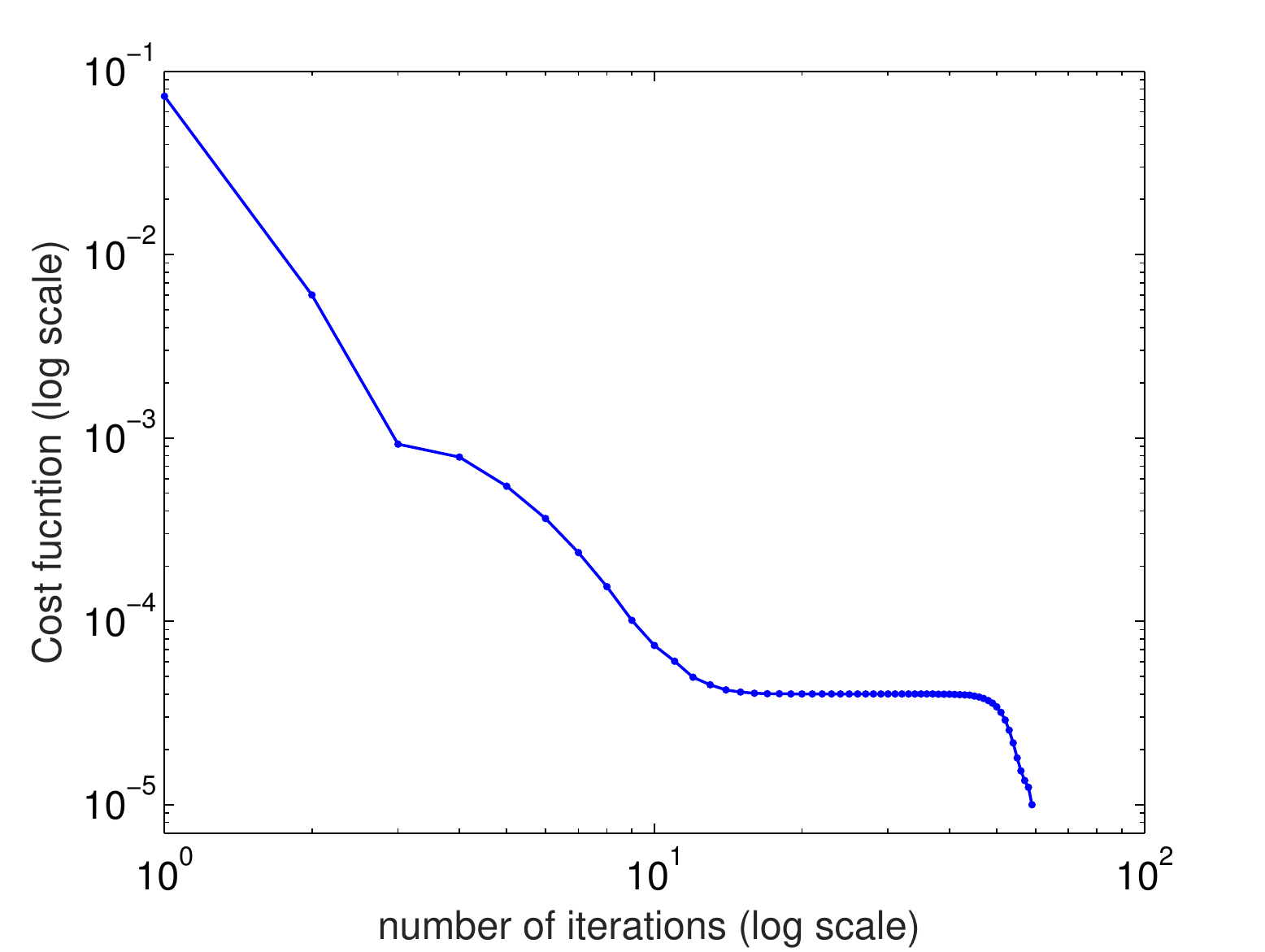}
\caption{ Convergence of the BFGS gradient descent algorithm for the minimization problem~\eqref{eq:optimal}.}
\label{fig:CONV}
\end{figure} 

We numerically observe that all the concentrations are well reconstructed and that the value of the optimal thickness $e_{\Phi^*} =  483.4022$ is close to the target thickness $e_{\Phi_{\rm sim}} = 483.4$. Unlike the external fluxes $\Phi_{\rm sim}$, the optimal fluxes$\Phi^*$ are not piecewise constant. These tests show that the uniqueness of a solution to the optimization problem \eqref{eq:optimal} can not be expected in general. We refer the reader to \cite{theseAthmane} for results obtained in the case of 
the control of PVD process for the fabrication of this film solar cells. 

\section{Conclusion}
In this work, we propose and analyze a one-dimensional model for the description of a PVD process. 
The evolution of the local concentrations of the different chemical species in the bulk of the growing layer is described via a system of cross-diffusion equations similar 
to the ones studied in~\cite{Burger,Jungel}. The growth of the thickness of the layer is related to the external fluxes of atoms that are absorbed at the surface of the film.

The existence of a global weak solution to the final system using the boundedness by entropy method under assumptions on the diffusion matrix of the system close to those needed in~\cite{Jungel} is established. 
In addition, the entropy density $h$ is required to be continuous (hence bounded) on the set $\overline{\cD} =  \left\{ u = (u_i)_{1 \leq i \leq n} \in \R_+^n, \; \sum_{i=1}^n u_i \leq 1\right\}$. 

We prove the existence of a solution to an optimization problem under the assumption that there exists a unique global weak solution to the obtained system, whatever the value of the external fluxes. 

Lastly, in the case when the entropy density is defined by $h(u) = \sum_{i=1}^n u_i \log u_i + (1-\rho_u) \log (1-\rho_u)$, we prove in addition that, when the external fluxes are constant and positive, the local concentrations 
converge in the long time to a constant profile at a rate which scales like $\mathcal{O}\left( \frac{1}{t} \right)$. 

\medskip

A discretization scheme, which is observed to be unconditionnaly stable, is introduced for the discretization of \eqref{eq:rescaled}. This scheme enables to preserve constraints \eqref{eq:constraint} at the discretized level.  

\medskip

We see this work as a preliminary step before tackling related problems in higher dimension, including surfacic diffusion effects. Besides, 
the proof of assumption (C1) remains an open question in general at least to our knowledge. Lastly, a nice theoretical question which is not tackled in this paper, but will be the object of future research, is 
the characterization of the set of reachable concentration profiles.

\medskip

\section{Acknowledgements}

We are very grateful to Eric Canc\`es and Tony Leli\`evre for very
helpful advice and discussions. We would like to thank Jean-François Guillemoles, Marie Jubeault, Torben Klinkert and Sana Laribi from IRDEP, who introduced us to the problem of modeling 
PVD processes for the fabrication of thin film solar cells. The EDF company is acknowledged for funding. We would also like to thank Daniel Matthes
for very helpful discussions on the theoretical part  {and the reviewers whose comments helped in very significantly improving the quality of the paper}.

\section{Appendix}\label{sec:appendix}

\subsection{Formal derivation of the diffusion model~\eqref{eq:system_all}}
We present in this section a simplified formal derivation of the cross-diffusion model \eqref{eq:system_all} from a one-dimensional microscopic lattice hopping model with size exclusion, 
in the same spirit than the one proposed in \cite{Burger}. 

We consider here a solid occupying the whole space $\R$ and discretize the domain using a uniform grid of step size $\Delta x >0$. At any time $t\in [0,T]$, we denote by $u_i^{k,t}$ the number of atoms of type $i$ 
($0\leq i \leq n$) in the $k^{th}$ interval $[k\Delta x, (k+1) \Delta x)$ ($k\in\Z$). Let $\Delta t>0$ denote a small enough time step. 
We assume that during the time interval $\Delta t$, an atom $i$ located in the $k^{th}$ interval can exchange its position with an atom of type $j$ ($j\neq i$) located in one of the two neighbouring intervals with 
probability $p_{ij} = p_{ji} >0$. In average, we obtain the following evolution equation for $u_i^{k,t}$: 
\begin{align*}
 u_i^{k,t+\Delta t} - u_i^{k,t} & = \sum_{0\leq j \neq i \leq n} p_{ij}\left( u_i^{k+1,t} u_j^{k,t} + u_i^{k-1,t} u_j^{k,t} - u_i^{k,t} u_j^{k+1,t} - u_i^{k,t} u_j^{k-1,t} \right)\\
 & = \sum_{0\leq j \neq i \leq n} p_{ij}\left[ u_j^{k,t}\left(u_i^{k+1,t} + u_i^{k-1,t} - 2 u_i^{k,t} \right) - u_i^{k,t} \left( u_j^{k+1,t} +  u_j^{k-1,t} - 2u_j^{k,t} \right) \right].\\
\end{align*}
This yields that
$$
\frac{u_i^{k,t+\Delta t} - u_i^{k,t}}{\Delta t}  = \frac{2 \Delta x^2}{\Delta t} \sum_{0\leq j \neq i \leq n} p_{ij} \left[ u_j^{k,t}\frac{u_i^{k+1,t} + u_i^{k-1,t} - 2 u_i^{k,t}}{2\Delta x^2} - u_i^{k,t} \frac{u_j^{k+1,t} +  u_j^{k-1,t} - 2u_j^{k,t}}{2\Delta x^2} \right].
$$
Choosing $\Delta t$ and $\Delta x$ so that these quantities satisfy a classical diffusion scaling $\frac{2 \Delta x^2}{\Delta t}  = \alpha >0$, denoting by $K_{ij}:= \alpha p_{ij}$ and letting the time step and grid size go to $0$, 
we formally obtain the following equation for the evolution of $u_i$ on the continuous level: 
$$
\partial_t u_i = \sum_{0\leq j \neq i \leq n} K_{ij}\left( u_j \Delta_x u_i - u_i \Delta_x u_j\right),  
$$
which is identical to the system of equations~\eqref{eq:system_all} introduced in the first section. Of course, this formal argument can be easily extended to any arbitrary dimension.

\subsection{Leray-Schauder fixed-point theorem}

\begin{theorem}[Leray-Schauder fixed-point theorem]\label{th:LS}
 Let $B$ be a Banach space and $S: B \times [0,1] \to B$ be a continuous map such that
 \begin{itemize}
  \item[(A1)] $S(x,0) = 0$ for each $x\in B$; 
  \item[(A2)] $S$ is a compact map; 
  \item[(A3)] there exists a constant $M>0$ such that for each pair $(x,\sigma) \in B \times [0,1]$ which satisfies $x = S(x,\sigma)$, we have $\|x\| <M$. 
 \end{itemize}
Then, there exists a fixed-point $y\in B$ satisfying $y= S(y,1)$. 
\end{theorem}

\bibliography{biblio}

\begin{thebibliography}{10}

\bibitem{Alikakos}
Nicholas~D Alikakos.
\newblock Lp bounds of solutions of reaction-diffusion equations.
\newblock {\em Communications in Partial Differential Equations},
  4(8):827--868, 1979.

\bibitem{Amann}
Herbert Amann et~al.
\newblock Dynamic theory of quasilinear parabolic equations. ii.
  reaction-diffusion systems.
\newblock {\em Differential and Integral Equations}, 3(1):13--75, 1990.

\bibitem{theseAthmane}
Athmane Bakhta.
\newblock Mathematical models and numerical simulation of photovoltaic devices.
\newblock {\em PhD thesis, in preparation}, 2017.

\bibitem{Boudin}
Laurent Boudin, B{\'e}r{\'e}nice Grec, and Francesco Salvarani.
\newblock A mathematical and numerical analysis of the maxwell-stefan diffusion
  equations.
\newblock {\em Discrete and Continuous Dynamical Systems-Series B},
  17(5):1427--1440, 2012.

\bibitem{Burger}
Martin Burger, Marco Di~Francesco, Jan-Frederik Pietschmann, and B{\"a}rbel
  Schlake.
\newblock Nonlinear cross-diffusion with size exclusion.
\newblock {\em SIAM Journal on Mathematical Analysis}, 42(6):2842--2871, 2010.

\bibitem{ChenJungel1}
Li~Chen and Ansgar J{\"u}ngel.
\newblock Analysis of a multidimensional parabolic population model with strong
  cross-diffusion.
\newblock {\em SIAM journal on mathematical analysis}, 36(1):301--322, 2004.

\bibitem{ChenJungel2}
Li~Chen and Ansgar J{\"u}ngel.
\newblock Analysis of a parabolic cross-diffusion population model without
  self-diffusion.
\newblock {\em Journal of Differential Equations}, 224(1):39--59, 2006.

\bibitem{Difrancesco}
Marco Di~Francesco and Jes{\'u}s Rosado.
\newblock Fully parabolic keller--segel model for chemotaxis with prevention of
  overcrowding.
\newblock {\em Nonlinearity}, 21(11):2715, 2008.

\bibitem{Dolbeault}
Jean Dolbeault, Bruno Nazaret, and Giuseppe Savar{\'e}.
\newblock A new class of transport distances between measures.
\newblock {\em Calculus of Variations and Partial Differential Equations},
  34(2):193--231, 2009.

\bibitem{AubinJungel}
Michael Dreher and Ansgar J{\"u}ngel.
\newblock Compact families of piecewise constant functions in lp (0, t; b).
\newblock {\em Nonlinear Analysis: Theory, Methods \& Applications},
  75(6):3072--3077, 2012.

\bibitem{Griepentrog}
Jens~Andr{\'e} Griepentrog and Lutz Recke.
\newblock Local existence, uniqueness and smooth dependence for nonsmooth
  quasilinear parabolic problems.
\newblock {\em Journal of Evolution Equations}, 10(2):341--375, 2010.

\bibitem{Painter2}
Thomas Hillen and Kevin~J Painter.
\newblock A user’s guide to pde models for chemotaxis.
\newblock {\em Journal of mathematical biology}, 58(1-2):183--217, 2009.

\bibitem{BONNANS}
Claude~Lemaréchal J.Frédéric~Bonnans, Jean Charles~Gilbert and Claudia
  Sagastizàbal.
\newblock {\em Numerical Optimization. Theoretical and Practical Aspects},
  volume~1.
\newblock Springer-Verlag Berlin Heidelberg, 2006.

\bibitem{JKO}
Richard Jordan, David Kinderlehrer, and Felix Otto.
\newblock The variational formulation of the fokker--planck equation.
\newblock {\em SIAM journal on mathematical analysis}, 29(1):1--17, 1998.

\bibitem{JungelStelzer2}
Ansgar Juengel and Ines~Viktoria Stelzer.
\newblock Entropy structure of a cross-diffusion tumor-growth model.
\newblock {\em Mathematical Models and Methods in Applied Sciences},
  22(07):1250009, 2012.

\bibitem{Jungel}
Ansgar J{\"u}ngel.
\newblock The boundedness-by-entropy method for cross-diffusion systems.
\newblock {\em Nonlinearity}, 28(6):1963, 2015.

\bibitem{JungelStelzer1}
Ansgar Jungel and Ines~Viktoria Stelzer.
\newblock Existence analysis of maxwell--stefan systems for multicomponent
  mixtures.
\newblock {\em SIAM Journal on Mathematical Analysis}, 45(4):2421--2440, 2013.

\bibitem{Kufner}
Konrad Horst~Wilhelm K{\"u}fner.
\newblock Invariant regions for quasilinear reaction-diffusion systems and
  applications to a two population model.
\newblock {\em Nonlinear Differential Equations and Applications NoDEA},
  3(4):421--444, 1996.

\bibitem{Ladyzenskaia}
Olga~A Ladyzenskaja and Vsevolod~A Solonnikov.
\newblock Nn ural ceva, linear and quasilinear equations of parabolic type,
  translated from the russian by s. smith. translations of mathematical
  monographs, vol. 23.
\newblock {\em American Mathematical Society, Providence, RI}, 63:64, 1967.

\bibitem{LeNguyen}
Dung Le and Toan~Trong Nguyen.
\newblock Everywhere regularity of solutions to a class of strongly coupled
  degenerate parabolic systems.
\newblock {\em Communications in Partial Differential Equations},
  31(2):307--324, 2006.

\bibitem{Lepoutre}
Thomas Lepoutre, Michel Pierre, and Guillaume Rolland.
\newblock Global well-posedness of a conservative relaxed cross diffusion
  system.
\newblock {\em SIAM Journal on Mathematical Analysis}, 44(3):1674--1693, 2012.

\bibitem{Mielke}
Matthias Liero and Alexander Mielke.
\newblock Gradient structures and geodesic convexity for reaction--diffusion
  systems.
\newblock {\em Phil. Trans. R. Soc. A}, 371(2005):20120346, 2013.

\bibitem{LionsMagenes}
Jacques~Louis Lions and Enrico Magenes.
\newblock {\em Non-homogeneous boundary value problems and applications},
  volume~1.
\newblock Springer Science \& Business Media, 2012.

\bibitem{PVD}
Donald~M Mattox.
\newblock {\em Handbook of physical vapor deposition (PVD) processing}.
\newblock William Andrew, 2010.

\bibitem{Painter1}
Kevin~J Painter.
\newblock Continuous models for cell migration in tissues and applications to
  cell sorting via differential chemotaxis.
\newblock {\em Bulletin of Mathematical Biology}, 71(5):1117--1147, 2009.

\bibitem{moving}
Jacobus~W Portegies and Mark~A Peletier.
\newblock Well-posedness of a parabolic moving-boundary problem in the setting
  of wasserstein gradient flows.
\newblock {\em arXiv preprint arXiv:0812.1269}, 2008.

\bibitem{Redlinger}
Reinhard Redlinger.
\newblock Invariant sets for strongly coupled reaction-diffusion systems under
  general boundary conditions.
\newblock {\em Archive for Rational Mechanics and Analysis}, 108(4):281--291,
  1989.

\bibitem{Stara}
Jana Star{\'a} and Oldrich John.
\newblock Some (new) counterexamples of parabolic systems.
\newblock {\em Commentationes Mathematicae Universitatis Carolinae},
  36(3):503--510, 1995.

\bibitem{JungelZamponi1}
Nicola Zamponi and Ansgar J{\"u}ngel.
\newblock Analysis of degenerate cross-diffusion population models with volume
  filling.
\newblock In {\em Annales de l'Institut Henri Poincare (C) Non Linear
  Analysis}. Elsevier, 2015.

\bibitem{JungelZamponi2}
Nicola Zamponi and Ansgar J{\"u}ngel.
\newblock Analysis of degenerate cross-diffusion population models with volume
  filling.
\newblock In {\em Annales de l'Institut Henri Poincare (C) Non Linear
  Analysis}. Elsevier, 2015.

\bibitem{MatthesZinsl}
Jonathan Zinsl and Daniel Matthes.
\newblock Transport distances and geodesic convexity for systems of degenerate
  diffusion equations.
\newblock {\em Calculus of Variations and Partial Differential Equations},
  54(4):3397--3438, 2015.

\end{thebibliography}
\end{document}